\documentclass[12pt]{article}
\usepackage[centertags]{amsmath}
\usepackage{amsfonts}
\usepackage{amssymb}
\usepackage{latexsym}
\usepackage{amsthm}
\usepackage{newlfont}
\usepackage{graphicx}
\usepackage{listings}
\usepackage{booktabs}
\usepackage{abstract}
\date{}

\newlength{\defbaselineskip}
\newcommand{\setlinespacing}[1]%
           {\setlength{\baselineskip}{#1 \defbaselineskip}}

\newcommand{\bR}{{\mathbb{R}}}

\newcommand{\N}{{\mathbb{N}}}

\newcommand{\actaqed}{\hfill $\actabox$}
{\medskip\noindent \textit{Proof of #1. }}%
{\actaqed \medskip}

\def\D{{\mathcal D}}

\def \Tr{\mathcal T}

\def \V{\mathcal V}
\def\R{{\mathbb R}}
\def\Z{\mathbb Z}

\def \T{\mathbb T}
\def \<{\langle}
\def\>{\rangle}
\def \L{\Lambda}
\def \al{\alpha}
\def\ga{\gamma}
\def \ep{\epsilon}
\def \e{\varepsilon}
\def \de{\delta}

\def \ff{\varphi}
\def\la{\lambda}

\def \sp{\operatorname{span}}

\def \sign{\operatorname{sign}}

\def\bx{\mathbf x}
\def\by{\mathbf y}
\def\bk{\mathbf k}
\def\bm{\mathbf m}
\def\bs{\mathbf s}
\def\bN{\mathbf N}
\def\bW{\mathbf W}

\def\fee{\varphi}

\def\U{{\mathcal U}}
\def\bt{\beta}
\def\mb{{\bar m}}
\def\DQn{\Delta Q_n}
\def\DQl{\Delta Q_l}

\newtheorem{Theorem}{Theorem}[section]
\newtheorem{Lemma}{Lemma}[section]

\newtheorem{Proposition}{Proposition}[section]
\newtheorem{Remark}{Remark}[section]

\newtheorem{Corollary}{Corollary}[section]
\numberwithin{equation}{section}

\newcommand{\be}{\begin{equation}}
\newcommand{\ee}{\end{equation}}

\begin{document}

\title{On the entropy numbers of the mixed smoothness function classes }
\author{V. Temlyakov \thanks{ University of South Carolina and Steklov Institute of Mathematics.  }}
\maketitle
\begin{abstract}
{Behavior of the entropy numbers of classes of multivariate functions with mixed smoothness is studied here. This problem has a long history and some fundamental problems in the area are still open. The main goal of this paper is to develop a new method of proving the upper bounds for the entropy numbers.   This method is based on recent developments of nonlinear approximation, in particular, on greedy approximation. This method consists of the following two steps strategy. At the first step we obtain bounds of the best $m$-term approximations with respect to a dictionary. At the second step we use general inequalities relating the entropy numbers to the best $m$-term approximations. For the lower bounds we use  the volume estimates method, which is a well known powerful method for proving the lower bounds for the entropy numbers. It was used in a number of previous papers.   }
\end{abstract}

\section{Introduction}  

Behavior of the entropy numbers of classes of multivariate functions with mixed smoothness is studied here. This problem has a long history and some fundamental problems in the area are still open. The main goal of this paper is to develop a new method of proving the upper bounds for the entropy numbers.   This method is based on recent developments of nonlinear approximation, in particular, on greedy approximation. This method consists of the following two steps strategy. At the first step we obtain bounds of the best $m$-term approximations with respect to a dictionary. At the second step we use general inequalities relating the entropy numbers to the best $m$-term approximations. For the lower bounds we use  the volume estimates method, which is a well known powerful method for proving the lower bounds for the entropy numbers. It was used in a number of previous papers.  
Taking into account the fact that there are fundamental open problems in the area, we give a detailed discussion of known results and of open problems. We also provide some comments on the techniques, which were used to obtain known results. Then we formulate our new results and compare them to the known results. 

Let $X$ be a Banach space and let $B_X$ denote the unit ball of $X$ with the center at $0$. Denote by $B_X(y,r)$ a ball with center $y$ and radius $r$: $\{x\in X:\|x-y\|\le r\}$. For a compact set $A$ and a positive number $\e$ we define the covering number $N_\e(A)$
 as follows
$$
N_\e(A) := N_\e(A,X) 
:=\min \{n : \exists y^1,\dots,y^n :A\subseteq \cup_{j=1}^n B_X(y^j,\e)\}.
$$
It is convenient to consider along with the entropy $H_\e(A,X):= \log_2 N_\e(A,X)$ the entropy numbers $\e_k(A,X)$:
$$
\e_k(A,X)  :=\inf \{\e : \exists y^1,\dots ,y^{2^k} \in X : A \subseteq \cup_{j=1}
^{2^k} B_X(y^j,\e)\}.
$$

Let $F_r(x,\alpha)$ be the univariate Bernoulli kernels 
$$
F_r(x,\alpha) := 1+2\sum_{k=1}^\infty k^{-r}\cos (kx-\alpha \pi/2).
$$
 For $\bx=(x_1,\dots,x_d)$ and $\alpha=(\alpha_1,\dots,\alpha_d)$ we define
$$
F_r({\bx},\alpha):= \prod_{i=1}^d F_r(x_i,\alpha_i)
$$
and
$$
\bW^r_{q,\alpha}:=\{f:f=F_r(\cdot,\alpha)\ast \fee,\quad \|\fee\|_q\le 1\}
$$
where $\ast$ means convolution. In the univariate case we use the notation $W^r_{q,\alpha}$.

It is well known that in the univariate case   
\begin{equation}\label{36.1}
\e_k(W^r_{q,\alpha},L_p)\asymp k^{-r}
\end{equation}
 holds for all $1\le q,p \le \infty$ and $r>(1/q-1/p)_+$. We note that condition $r>(1/q-1/p)_+$ is a necessary and sufficient condition for compact embedding of $W^r_{q,\alpha}$ into $L_p$. Thus (\ref{36.1}) provides a complete description of the rate of $\e_k(W^r_{q,\alpha},L_p)$ in the univariate case. We point out that (\ref{36.1}) shows that the rate of decay of $\e_k(W^r_{q,\alpha},L_p)$ depends only on $r$ and does not depend on $q$ and $p$. In this sense the strongest upper bound (for $r>1$) is $\e_k(W^r_{1,\alpha},L_\infty) \ll k^{-r}$ and the strongest lower bound is $\e_k(W^r_{\infty,\alpha},L_1)\gg k^{-r}$. 

There are different generalizations of classes $W^r_{q,\alpha}$ to the case of multivariate functions. In this section we only discuss known results for classes $\bW^r_{q,\alpha}$ of functions with bounded mixed derivative. For further discussions see \cite{Tbook}, Chapter 3 and \cite{DTU}.    

The following two theorems are from \cite{TE1} and \cite{TE2}.

\begin{Theorem}\label{Wup} For $r>1$ and $1<q,p<\infty$ one has
$$
\e_k(\bW^r_{q,\alpha},L_p) \ll k^{-r}(\log k)^{r (d-1)}.
$$
\end{Theorem}

\begin{Theorem}\label{Wlo} For $r>0$ and $1\le q <\infty$ one has
$$
\e_k(\bW^r_{q,\alpha},L_1) \gg k^{-r}(\log k)^{r (d-1)}.
$$
\end{Theorem}

The problem of estimating $\e_k(\bW^r_{q,\alpha},L_p)$ has a long history. The first result on the right order of $\e_k(\bW^r_{2,\alpha},L_2)$ was obtained by Smolyak \cite{Smo}. Later (see \cite{TE1}, \cite{TE2} and theorems above) it was established that
\begin{equation}\label{36.2}
\e_k(\bW^r_{q,\alpha},L_p)\asymp k^{-r}(\log k)^{r(d-1)} 
\end{equation}
holds for all $1<q,p<\infty$, $r>1$. The case $1<q=p<\infty$, $r>0$ was established by Dinh Dung \cite{DD}. Belinskii \cite{Bel} extended (\ref{36.2}) to the case $r>(1/q-1/p)_+$ when $1<q,p<\infty$. 

It is known in approximation theory (see \cite{TBook}) that investigation of asymptotic characteristics of classes $\bW^r_{q,\alpha}$ in $L_p$ becomes more difficult when $q$ or $p$ takes value $1$ or $\infty$ than when $1<q,p<\infty$. It turns out to be the case for $\e_k(\bW^r_{q,\alpha},L_p)$ too. It was discovered that in some of these extreme cases ($q$ or $p$ equals $1$ or $\infty$) relation (\ref{36.2}) holds and in other cases it does not hold. We describe the picture in detail. It was proved in \cite{TE2} that (\ref{36.2}) holds for $p=1$, $1<q<\infty$, $r>0$. It was also proved that (\ref{36.2}) holds for $p=1$, $q=\infty$ (see \cite{Bel} for $r>1/2$ and  \cite{KTE2} for $r>0$). Summarizing, we state that (\ref{36.2}) holds for $1<q,p<\infty$ and $p=1$, $1<q\le\infty$ for all $d$ (with appropriate  restrictions on $r$). This easily implies that (\ref{36.2}) also holds for $q=\infty$, $1\le p<\infty$. For all other pairs $(q,p)$, namely, for $p=\infty$, $1\le q\le\infty$ and $q=1$, $1\le p\le \infty$ the rate of $\e_k(\bW^r_{q,\alpha},L_p)$ is not known in the case $d>2$. It is an outstanding open problem. 

In the case $d=2$ this problem is essentially solved. We now cite the corresponding results. The first result on the right order of $\e_k(\bW^r_{q,\alpha},L_p)$ in the case $p=\infty$ was obtained by Kuelbs and Li \cite{KL} for $q=2$, $r=1$. It was proved in \cite{TE3} that
\begin{equation}\label{36.3}
\e_k(\bW^r_{q,\alpha},L_\infty)\asymp k^{-r}(\log k)^{r+1/2}
\end{equation}
holds for $1<q<\infty$, $r>1$. We note that the upper bound in (\ref{36.3}) was proved under condition $r>1$ and the lower bound in (\ref{36.3}) was proved under condition $r>1/q$. Belinskii \cite{Bel} proved the upper bound in (\ref{36.3}) for $1<q<\infty$ under condition $r>\max(1/q,1/2)$.   Relation (\ref{36.3}) for $q=\infty$ under assumption $r>1/2$ was proved in  \cite{TE4}. 

The case $q=1$, $1\le p\le \infty$ was settled by Kashin and Temlyakov \cite{KaTe03}. The authors proved that 
\begin{equation}\label{36.3'} 
\e_k(\bW^r_{1,\alpha},L_p)\asymp k^{-r}(\log k)^{r+1/2}
\end{equation}
holds for $1\le p<\infty$, $r>\max(1/2,1-1/p)$ and
\begin{equation}\label{36.4} 
\e_k(\bW^r_{1,0},L_\infty)\asymp k^{-r}(\log k)^{r+1},\quad r>1.
\end{equation}

Let us make an observation on the base of the above discussion. In the univariate case the entropy numbers $\e_k(W^r_{q,\alpha},L_p)$ have the same order of decay with respect to $k$ for all pairs $(q,p)$, $1\le q,p\le\infty$. In the case $d=2$ we have three different orders of decay of $\e_k(\bW^r_{q,\alpha},L_p)$ which depend on the pair $(q,p)$. For instance, in the case $1<q,p<\infty$ it is $k^{-r}(\log k)^r$, in the case $q=1$, $1<p<\infty$, it is $k^{-r}(\log k)^{r+1/2}$ and in the case $q=1$, $p=\infty$ it is $k^{-r}(\log k)^{r+1}$. 

We discussed above results on the right order of decay of the entropy numbers. Clearly, each order relation $\asymp$ is a combination of the upper bound $\ll$ and the matching lower bound $\gg$. We now briefly discuss methods that were used for proving upper and lower bounds. The upper bounds in Theorem  \ref{Wup} were proved by the standard method of reduction by discretization to estimates of the entropy numbers of finite-dimensional sets.  
Here results of \cite{H}, \cite{M} or \cite{Schu} are applied.   It is clear from the above discussion that it was sufficient to prove the lower bound in (\ref{36.2}) in the case $p=1$. The proof of this lower bound (see Theorem  \ref{Wlo}) is more difficult and is based on nontrivial estimates of the volumes of the sets of Fourier coefficients of bounded trigonometric polynomials. Theorem \ref{T2.1KaTe} (see below) plays a key role in this method. 

An analogue of the upper bound in (\ref{36.3}) for any $d$ was obtained by Belinskii \cite{Bel}: for $q>1$ and $r>\max(1/q,1/2)$ we have
\begin{equation}\label{36.4'} 
\e_k(\bW^r_{q,\alpha},L_\infty)\ll k^{-r}(\log k)^{(d-1)r+1/2}.
\end{equation}
That proof is based on Theorem \ref{Theorem 32.3} (see below). 

  Kuelbs and Li \cite{KL} discovered the fact that there is a tight relationship between small ball problem  and the behavior of the entropy $H_\e(\bW^1_{2,\alpha},L_\infty)$. Based on results obtained  by Livshits and Tsirelson \cite{LiTs}, by Bass \cite{Bass}, and  by Talagrand \cite{TaE} for the small ball problem, they proved
\begin{equation}\label{36.6}
\e_k(\bW^1_{2,\alpha},L_\infty) \asymp k^{-1}(\ln k)^{3/2}.
\end{equation}
Proof of the most difficult part of (\ref{36.6}) -- the lower bound -- is based on a special inequality, known now as the Small Ball Inequality, for the Haar polynomials proved by Talagrand \cite{TaE} (see \cite{TE5} for a simple proof). 

 We discussed above known results on the rate of decay of  $\e_k(\bW^r_{q,\alpha},L_p)$. In the case $d=2$ the picture is almost complete. In the case $d>2$ the situation is fundamentally different. The problem of the right order of decay of $\e_k(\bW^r_{q,\alpha},L_p)$ is still open for $q=1$, $1\le p\le \infty$ and $p=\infty$, $1\le q\le\infty$. In particular, it is open in the case $q=2$, $p=\infty$, $r=1$ that is related to the small ball problem. We discuss in more detail the case $p=\infty$, $1\le q\le \infty$. We pointed out above that in the case $d=2$  the proof of lower bounds (the most difficult part) was based on the Small Ball Inequalities for the Haar system for $r=1$ and for the trigonometric system for all $r$. The existing conjecture is that 
\begin{equation}\label{36.12}
\e_k(\bW^r_{q,\alpha},L_\infty) \asymp k^{-r}(\ln k)^{(d-1)r+1/2},\quad 1<q<\infty,
\end{equation}
for large enough $r$. The upper bound in (\ref{36.12}) follows from (\ref{36.4'}). It is known that the corresponding lower bound in (\ref{36.12}) would follow from the $d$-dimensional version of the Small Ball Inequality for the trigonometric system. 

The main goal of this paper is to develop new techniques for proving upper bounds for the entropy numbers. We consider here slightly more general classes than classes $\bW^r_{q,\alpha}$. Let $\bs=(s_1,\dots,s_d)$ be a vector with nonnegative integer coordinates ($\bs \in \Z^d_+$) and
$$
\rho(\bs):= \{\bk=(k_1,\dots,k_d)\in \Z^d_+ : [2^{s_j-1}]\le |k_j|<2^{s_j},\quad j=1,\dots,d\}
$$
where $[a]$ denotes the integer part of a number $a$.  Define for $f\in L_1$
$$
\delta_\bs(f) := \sum_{\bk\in\rho(\bs)} \hat f(\bk) e^{i(\bk,\bx)},
$$
and
$$
f_l:=\sum_{\|\bs\|_1=l}\delta_\bs(f), \quad l\in \N_0,\quad \N_0:=\N\cup \{0\}.
$$
  Consider the class (see \cite{VT152})
$$
\bW^{a,b}_q:=\{f: \|f_l\|_q \le 2^{-al}(\bar l)^{(d-1)b}\},\quad \bar l:=\max(l,1).
$$
Define
$$
\|f\|_{\bW^{a,b}_q} := \sup_l \|f_l\|_q 2^{al}(\bar l)^{-(d-1)b}.
$$
It is well known that the class $\bW^r_{q,\alpha}$ is embedded in the class $\bW^{r,0}_q$ for $1<q<\infty$. 
Classes $\bW^{a,b}_q$ provide control of smoothness at two scales: $a$ controls the power type smoothness and $b$ controls the logarithmic scale smoothness. Similar classes with the power and logarithmic scales of smoothness are studied in the recent book of Triebel \cite{Tr}.
Here is one more class, which is equivalent to $\bW^{a,b}_q$ in the case $1<q<\infty$ (see \cite{VT152}). 
Consider a class ${\bar \bW}^{a,b}_q$, which consists of functions $f$ with a representation (see Subsection 2.2 below for the definition of $\Tr(Q_n)$)
$$
f=\sum_{n=1}^\infty t_n, \quad t_n\in \Tr(Q_n), \quad \|t_n\|_q \le 2^{-an} n^{b(d-1)}.
$$
  In the case $q=1$ classes ${\bar \bW}^{a,b}_1$ are wider than $\bW^{a,b}_1$.  

The main results of the paper are the following  theorems in the case $d=2$ for the extreme values of $q=1$ and $q=\infty$. First, we formulate two theorems for the case $q=1$.

\begin{Theorem}\label{T2.1} Let $1\le p <\infty$ and $a>1-1/p$. Then for $d=2$
\be\label{2.0}
\e_k(\bW^{a,b}_1,L_p) \asymp \e_k(\bar\bW^{a,b}_1,L_p) \asymp k^{-a} (\log k)^{a+b}.
\ee
\end{Theorem}

\begin{Theorem}\label{T2.3} Let $d=2$ and $a>1$. Then
\be\label{2.0'}
\e_k(\bW^{a,b}_1,L_\infty) \asymp \e_k(\bar\bW^{a,b}_1,L_\infty) \asymp k^{-a} (\log k)^{a+b+1/2}.
\ee
\end{Theorem}

Second, we formulate three theorems for the case $q=\infty$.

\begin{Theorem}\label{T1.5} We have for all $d\ge 2$
\be\label{1.11}
\e_k(\bW^{a,b}_\infty,L_1) \gg k^{-a} (\log k)^{(d-1)(a+b)-1/2}.
\ee
\end{Theorem}

\begin{Theorem}\label{T1.6} We have for  $d= 2$, $1\le p<\infty$, $a\ge \max(1/2,1-1/p)$
\be\label{1.12}
\e_k(\bW^{a,b}_\infty,L_p) \asymp \e_k(\bar\bW^{a,b}_\infty,L_p) \asymp  k^{-a} (\log k)^{a+b-1/2}.
\ee
\end{Theorem}

\begin{Theorem}\label{T1.7} We have for  all $d\ge 2$, $1\le q\le \infty$,  $a>0$
\be\label{1.13}
\e_k(\bW^{a,b}_q,L_q) \asymp \e_k(\bar\bW^{a,b}_q,L_q) \asymp  k^{-a} (\log k)^{(d-1)(a+b)}.
\ee
\end{Theorem}

Let us make some comments on Theorem \ref{T2.1}. As we already mentioned above classes $\bW^{r,0}_1$ are close to classes $\bW^r_{1,\alpha}$ but they are different. We show that they are different even in the sense of asymptotic behavior of their entropy numbers. We point out that the right order of $\e_k(\bW^r_{1,\alpha},L_p)$ is not known for $d>2$. We confine ourselves to the case $d=2$. It is proved in \cite{KaTe03} that for $r> \max(1/2,1-1/p)$
\be\label{2.7}
\e_k(\bW^r_{1,\alpha},L_p) \asymp k^{-r}(\log k)^{r+1/2}, \qquad 1\le p<\infty.
\ee
Theorem \ref{T2.1} gives for $r>1-1/p$
\be\label{2.8}
\e_k(\bW^{r,0}_1,L_p) \asymp k^{-r}(\log k)^{r}, \qquad 1\le p<\infty.
\ee
This shows that in the sense of the entropy numbers class $\bW^{r,0}_1$ is smaller than $\bW^r_{1,\alpha}$. It is interesting to compare (\ref{2.7}) and (\ref{2.8}) with the known estimates in the case $1<q,p<\infty$
\be\label{2.9}
\e_k(\bW^{r}_{q,\alpha},L_p) \asymp \e_k(\bW^{r,0}_q,L_p) \asymp k^{-r}(\log k)^{r}, \qquad 1\le p<\infty.
\ee
Relation (\ref{2.9}) is for the case $d=2$. The general case of $d$ is also known in this case (see (\ref{36.2}) and its discussion above and also see Section 3.6 of \cite{Tbook} for the corresponding results and historical comments). Relations (\ref{2.8}) and (\ref{2.9}) show that in the sense of entropy numbers the class $\bW^{r,0}_1$ behaves as a limiting case of classes $\bW^r_{q,\alpha}$ when $q\to 1$. 

The proof of upper bounds in Theorems \ref{T2.1} and \ref{T2.3} is based on greedy approximation technique. It is a new and powerful technique. In particular, Theorem \ref{T2.3} gives the same upper bound as in (\ref{36.4'}) for the class $\bar\bW^{r,0}_1$, which is  wider than any of the classes $\bW^r_q$, $q>1$, from (\ref{36.4'}).  
 In Section 7 we develop mentioned above new technique, which is based on nonlinear $m$-term approximations, to prove the following result.

\begin{Theorem}\label{T4.6} Let $1<q\le 2$ and  $a>1/q$. Then
\be\label{4.18}
  \e_k(\bW^{a,b}_q,L_\infty) \ll k^{-a} (\log k)^{(d-1)(a+b)+1/2}.
\ee
\end{Theorem}

In particular, Theorem \ref{T4.6} implies (\ref{36.4'}). 

Theorem \ref{T1.6} discovers an interesting new phenomenon. Comparing (\ref{1.12})  with (\ref{1.13}), we see that the entropy numbers of the class $\bar \bW^{a,b}_\infty$ in the $L_p$ space have different rate of decay in cases $1\le p<\infty$ and $p=\infty$. We note that in the proof of the upper bounds in this new phenomenon we use the Riesz products for the hyperbolic crosses. This technique works well in the case $d=2$ but we do not know how to extend it to the general case $d>2$. This difficulty is of the same nature as the corresponding difficulty in generalizing the Small Ball Inequality from $d=2$ to $d>2$ (see \cite{Tbook}, Ch. 3, for further discussion). 
We already mentioned above that in studying the entropy numbers of function classes
the discretization technique is useful. Classically,
the Marcinkiewicz theorem serves as a powerful tool for discretizing the
$L_p$-norm of a trigonometric polynomial. It works well in the multivariate case for trigonometric polynomials with frequencies from a parallelepiped. 
However, there
is no analog of Marcinkiewicz' theorem for hyperbolic cross polynomials (see \cite{KaTe03} and \cite{DTU}, Section 2.5, for a discussion). Thus, in Sections 5--7
we develop a new technique for estimating the entropy numbers of the unit balls of the hyperbolic cross polynomials. The most interesting results are obtained in the 
dimension $d=2$. It would be very interesting to extend these results to the case $d>2$. It is a challenging open problem. 

Finally, we emphasize that in the case $1<q<\infty$, when the classes $\bW^r_{q,\al}$ are embedded in the classes $\bW^{r,0}_q$, the new technique, developed in this paper, provides all known upper bounds. In the case $p<\infty$ Theorem \ref{T4.2} gives the upper bounds in (\ref{36.2}), and in the case $p=\infty$ Theorem \ref{T4.6} gives the upper bounds in (\ref{36.4'}).

\section{Known results}
\subsection{General inequalities}
For the reader's convenience we collect in this section known results, which will be used in this paper. The reader can find results of this subsection, except Theorem \ref{NI}, and their proofs in \cite{Tbook}, Chapter 3.
\begin{Proposition}\label{Corollary 31.1} Let $A\subset Y$, and let $Y$ be a subspace of $X$. Then
$$
N_\e(A,X) \ge N_{2\e}(A,Y).
$$
\end{Proposition}

Let us consider the space ${\mathbb R}^D$ equipped with different norms, say, norms $\|\cdot\|_X$ and $\|\cdot\|_Y$. For a Lebesgue measurable set $E\in {\mathbb R}^D$ we denote its Lebesgue measure by $vol(E):=vol_D(E)$.
\begin{Theorem}\label{Theorem 32.1} For any two norms $X$ and $Y$ and any $\e>0$ we have
\begin{equation}\label{32.1}
\frac{1}{\e^D}\frac{vol(B_Y)}{vol(B_X)} \le N_\e(B_Y,X) \le \frac{vol(B_Y(0,2/\e)\oplus B_X)}{vol(B_X)}. 
\end{equation}
\end{Theorem}
Let us formulate one immediate corollary of Theorem \ref{Theorem 32.1}.
\begin{Corollary}\label{Corollary 32.1} For any $D$-dimensional real Banach space $X$ we have
$$
\e^{-D} \le N_\e(B_X,X) \le (1+2/\e)^D,
$$
and, therefore,
$$
\e_k(B_X,X) \le 3(2^{-k/D}).
$$
\end{Corollary}

Let $|\cdot| :=\|\cdot\|_2$ denote the $\ell^D_2$ norm and let $B^D_2$ be a unit ball in $\ell^D_2$. Denote $S^{D-1}$ the boundary of $B^D_2$. We define by $d\sigma(x)$ the normalized $(D-1)$-dimensional measure on $S^{D-1}$. Consider another norm $\|\cdot\|$ on $\bR^D$ and denote by $X$ the $\bR^D$ equipped with $\|\cdot\|$.  
\begin{Theorem}\label{Theorem 32.3} Let $X$ be $\bR^D$ equipped with $\|\cdot\|$ and
$$
M_X := \int_{S^{D-1}} \|x\|d\sigma(x).
$$
Then we have
$$
 \e_k(B^D_2,X) \ll M_X \left\{\begin{array}{ll}   (D/k)^{1/2}, & k\le D\\
 2^{-k/D} , &  k\ge D.\end{array} \right.
$$
\end{Theorem}
 The following Nikol'skii-type inequalities are known (see \cite{Tmon}, Chapter 1, Section 2).
\begin{Theorem}\label{NI} Let $1\le q< p <\infty$. For any $t\in \Tr(Q_n)$ (see Subsection 2.2 below for the definition of $\Tr(Q_n)$) we have
$$
\|t\|_p \le C(q,p,d)2^{\bt n} \|t\|_q,\quad \bt:=1/q-1/p.
$$
\end{Theorem}

\subsection{Volume estimates}
  Denote for a natural number $n$
$$
Q_n := \cup_{\|\bs\|_1\le n}\rho(\bs); \qquad \Delta Q_n := Q_n\setminus Q_{n-1} = \cup_{\|\bs\|_1=n}\rho(\bs)
$$
with $\|\bs\|_1 = s_1+\dots+s_d$ for $\bs\in \Z^d_+$. We call a set $\Delta Q_n$ {\it hyperbolic layer}. For a set $\L \subset \Z^d$ denote
$$
\Tr(\L) := \{f\in L_1:  \hat f(\bk)=0, \bk\in \mathbb Z^d\setminus \L\},\quad \Tr(\L)_p :=\{f\in\Tr(\L): \|f\|_p\le 1\}.
$$
For a finite set $\L$ we assign to each $f = \sum_{\bk\in \L} \hat f(\bk) e^{i(\bk,\bx)}\in \Tr(\L)$ a vector
$$
A(f) := \{(\text{Re}\hat f(\bk), \text{Im}\hat f(\bk)),\quad \bk\in \L\} \in \R^{2|\L|}
$$
where $|\L|$ denotes the cardinality of $\L$ and define
$$
B_\L(L_p) := \{A(f) : f\in \Tr(\L)_p\}.
$$
The volume estimates of the sets $B_\L(L_p)$ and related questions have been studied in a number of papers: the case $\L=[-n,n]$, $p=\infty$ in \cite{K1}; the case $\L=[-N_1,N_1]\times\cdots\times[-N_d,N_d]$, $p=\infty$ in \cite{TE2}, \cite{T3}.   In the case $\L = \Pi(\bN,d) := [-N_1,N_1]\times\cdots\times[-N_d,N_d]$, $\bN:=(N_1,\dots,N_d)$, the following estimates are known.
\begin{Theorem}\label{T2.1KaTe} For any $1\le p\le \infty$ we have
$$
(vol(B_{\Pi(\bN,d)}(L_p)))^{(2|\Pi(\bN,d)|)^{-1}} \asymp |\Pi(\bN,d)|^{-1/2},
$$
with constants in $\asymp$ that may depend only on $d$.
\end{Theorem}
We note that the most difficult part of Theorem \ref{T2.1KaTe} is the lower estimate for $p=\infty$. The corresponding estimate was proved in the case $d=1$ in \cite{K1} and in the general case in \cite{TE2} and \cite{T3} by a method different from the one in \cite{K1}. The upper estimate for $p=1$ in Theorem \ref{T2.1KaTe} can be easily reduced to the volume estimate for an octahedron (see, for instance \cite{T4}). In the case $p=2$ Theorem \ref{T2.1KaTe} is a direct corollary of the well known estimates of the volume of the Euclidean unit ball. 

The case of arbitrary $\L$ and $p=1$ was studied in \cite{KT1}. The results of \cite{KT1} imply the following estimate.
\begin{Theorem}\label{T2.2KaTe} For any finite set $\L\subset \Z^d$ and any $1\le p\le 2$ we have
$$
vol(B_\L(L_p))^{(2|\L|)^{-1}} \asymp |\L|^{-1/2}.
$$
\end{Theorem}

The following result was obtained in \cite{KaTe03}.
\begin{Theorem}\label{T2.4KaTe} Let $\L$ have the form
$\L = \cup_{\bs\in S}\rho(\bs)$, $S\subset \Z^d_+$ is a finite set. Then for any $1\le p<\infty$ we have
$$
vol(B_\L(L_p))^{(2|\L|)^{-1}} \asymp |\L|^{-1/2}.
$$
\end{Theorem}
In particular, Theorem \ref{T2.4KaTe} implies
  for $d=2$ and $1\le p<\infty$ that
\be\label{2.9KaTe}
(vol(B_{\Delta Q_n}(L_p)))^{(2|\Delta Q_n|)^{-1}} \asymp |\Delta Q_n|^{-1/2} \asymp (2^nn)^{-1/2}.
\ee

The following result was obtained in \cite{KaTe03}. Denote $D:= 2|\Delta Q_n|$.

\begin{Theorem}\label{T2.5KaTe} In the case $d=2$ we have
\be\label{2.7KaTe}
(vol(B_{\Delta Q_n}(L_\infty)))^{1/D} \asymp (2^nn^2)^{-1/2}.
\ee
\end{Theorem}

The following lemma from \cite{KaTe03} is an important ingredient of analysis in this paper.
For the reader's convenience we give a proof of this lemma here.

\begin{Lemma}\label{L2.1KaTe} Let $\L \subseteq [-2^n,2^n]^d$ and $N:=2|\L|$. Then
$$
(vol(B_\L(L_\infty)))^{1/N} \ge C(d) (Nn)^{-1/2}.
$$
\end{Lemma}
\begin{proof} We use the following result of E. Gluskin \cite{G}. 
\begin{Theorem}\label{T2.6KaTe} Let $Y=\{\by_1,\dots,\by_M\} \subset \R^N$, $\|\by_i\|=1$, $i=1,\dots,M$ and 
$$
W(Y) := \{\bx\in \R^N:|(\bx,\by_i)| \le 1,\quad i=1,\dots,M\}.
$$
Then
$$
(vol(W(Y)))^{1/N} \ge C(1+\log (M/N))^{-1/2}.
$$
\end{Theorem}

Consider the following lattice on the $\T^d$:
$$
G_n:= \{ \bx(l)=(l_1,\dots,l_d)\pi 2^{-n-1}, \quad 1\le l_j \le 2^{n+2}, \quad l_j\in \mathbb N,\quad j=1,\dots,d\}.
$$
It is clear that $|G_n|=2^{d(n+2)}$. It is well known (see \cite{TBook}, Ch.2, Theorem 2.4) that for any $f\in \Tr([-2^n,2^n]^d)$ one has
$$
\|f\|_\infty \le C_1(d) \max_{\bx\in G_n} |f(\bx)|.
$$
Thus, for any $\L \subseteq [-2^n,2^n]^d$ we have
\be\label{2.10KaTe}
\{A(f):f\in \Tr(\L),\quad |f(\bx)|\le C_1(d)^{-1},\quad \bx\in G_n\} \subseteq B_\L(L_\infty). 
\ee
Further
$$
|f(\bx)|^2 = |\sum_{\bk\in\L}\hat f(\bk) e^{i(\bk,\bx)}|^2 = 
$$
$$
\left(\sum_{\bk\in\L}\text{Re}\hat f(\bk) \cos(\bk,\bx) - \text{Im}\hat f(\bk) \sin(\bk,\bx)\right)^2 
$$
$$
+\left(\sum_{\bk\in\L}\text{Re}\hat f(\bk) \sin(\bk,\bx) + \text{Im}\hat f(\bk) \cos(\bk,\bx)\right)^2 .
$$
We associate with each point $\bx\in G_n$ two vectors $\by^1(\bx)$ and $\by^2(\bx)$ from $\R^N$:
$$
\by^1(\bx) := \{(\cos(\bk,\bx),-\sin(\bk,\bx)),\quad \bk\in \L\},
$$
$$
 \by^2(\bx) := \{(\sin(\bk,\bx),\cos(\bk,\bx)),\quad \bk\in \L\}.
$$
Then
$$
\|\by^1(\bx)\|^2 =\|\by^2(\bx)\|^2 = |\L|
$$
and
$$
|f(\bx)|^2 = (A(f),\by^1(\bx))^2 +(A(f),\by^2(\bx))^2.
$$
It is clear that the condition $|f(\bx)| \le C_1(d)^{-1}$ is satisfied if
$$
|(A(f),\by^i(\bx))| \le 2^{-1/2}C_1(d)^{-1}, \quad i=1,2.
$$
Let now
$$
Y:=\{\by^i(\bx)/\|\by^i(\bx)\|,\quad \bx\in G_n,\quad i=1,2\}.
$$
Then $M=2^{d(n+2)+1}$ and by Theorem \ref{T2.6KaTe}
\be\label{2.11KaTe}
(vol(W(Y)))^{1/N} \gg (1+\log (M/N))^{-1/2} \gg n^{-1/2}. 
\ee
Using that the condition
$$
|(A(f),\by^i(\bx))|\le 1
$$
is equivalent to the condition
$$
|(A(f),\by^i(\bx)/\|\by^i(\bx)\|)| \le (N/2)^{-1/2}
$$ 
we get from (\ref{2.10KaTe}) and (\ref{2.11KaTe})
$$
(vol(B_\L(L_\infty)))^{1/N} \gg (Nn)^{-1/2}.
$$
This completes the proof of Lemma \ref{L2.1KaTe}
\end{proof}

\section{New lower bounds. The volumes technique}

In this section we prove lower bounds in Theorems \ref{T2.1} -- \ref{T1.7} from the Introduction.

 {\bf Proof of lower bounds in Theorems \ref{T2.1} and \ref{T1.7}.}  The lower bound in Theorem \ref{T2.1} follows from the lower bound in Theorem \ref{T1.7} with $q=1$. 
We prove the lower bounds for the $\e_k(\bW^{a,b}_q,L_q)$ with $1\le q\le \infty$ and any $d$. This lower bound is derived from the well known simple inequality (see Corollary \ref{Corollary 32.1} above)
\be\label{2.2}
N_\e(B_X,X) \ge \e^{-D}
\ee
for any $D$-dimensional real Banach space $X$. Consider as a Banach space $X$ the $\Tr(\Delta Q_n)$ with $L_q$ norm. Clearly, it can be seen as a $D$-dimensional real Banach space with $D=2|\Delta Q_n|$. It follows from the definition of $\bW^{a,b}_q$ that
\be\label{2.3}
2^{-an}n^{b(d-1)}\Tr(\Delta Q_n)_q \subset \bW^{a,b}_q.
\ee
Take $k=2|\Delta Q_n|$. Then (\ref{2.2}) implies that 
\be\label{2.4}
\e_k(\Tr(\Delta Q_n)_q, L_q\cap \Tr(\Delta Q_n)) \gg 1.
\ee
We now use one more well known fact from the entropy theory -- Proposition \ref{Corollary 31.1}. This and inequality (\ref{2.4}) imply
\be\label{2.6}
\e_k(\Tr(\Delta Q_n)_q, L_q) \gg 1.
\ee
Taking into account (\ref{2.3}) and the fact $k\asymp 2^nn^{d-1}$ we derive from (\ref{2.6}) the required lower bound for the $\e_k(\bW^{a,b}_q,L_q)$.

The lower bounds in Theorems \ref{T2.1} and \ref{T1.7} are proved.

{\bf Proof of lower bounds in Theorem \ref{T2.3}.}   We  prove the lower bound for $\e_k(\bW^{a,b}_1,L_\infty)$. This proof is somewhat similar to the proof of lower bounds in Theorem \ref{T2.1}. Instead of (\ref{2.2}) we now use the inequality (see Theorem \ref{Theorem 32.1} above)
\be\label{2.11}
N_\e(B_Y,X) \ge \e^{-D} \frac{vol(B_Y)}{vol(B_X)}
\ee
 with $B_Y := B_{\Delta Q_n}(L_1)$ and $B_X := B_{\Delta Q_n}(L_\infty)$.   It follows from the definition of $\bW^{a,b}_1$ that
\be\label{2.3'}
2^{-an}n^{b(d-1)}\Tr(\Delta Q_n)_1 \subset \bW^{a,b}_1.
\ee
Take $k=2|\Delta Q_n|$. Then (\ref{2.11}), Theorem \ref{T2.5KaTe}, and (\ref{2.9KaTe}) imply that 
\be\label{2.14}
\e_k(\Tr(\Delta Q_n)_1, L_\infty\cap \Tr(\Delta Q_n)) \gg n^{1/2}.
\ee
   Proposition \ref{Corollary 31.1} and inequality (\ref{2.14}) imply
\be\label{2.6'}
\e_k(\Tr(\Delta Q_n)_1, L_\infty) \gg n^{1/2}.
\ee
Taking into account (\ref{2.3'}) and the fact $k\asymp 2^nn^{d-1}$ we derive from (\ref{2.6'}) the required lower bound for the $\e_k(\bW^{a,b}_1,L_\infty)$.

The lower bounds in Theorem \ref{T2.3} are proved.

{\bf Proof of  Theorem \ref{T1.5}.}   We  prove the lower bound for $\e_k(\bW^{a,b}_\infty,L_1)$. This proof goes along the lines of the above proof of lower bounds in Theorem \ref{T2.3}. We use (\ref{2.11}) with $B_X := B_{\Delta Q_n}(L_1)$ and $B_Y := B_{\Delta Q_n}(L_\infty)$.   It follows from the definition of $\bW^{a,b}_\infty$ that
\be\label{3.12}
2^{-an}n^{b(d-1)}\Tr(\Delta Q_n)_\infty \subset \bW^{a,b}_\infty.
\ee
Take $k=2|\Delta Q_n|$. Then (\ref{2.11}), Lemma \ref{L2.1KaTe} with $\L=\DQn$, and Theorem \ref{T2.2KaTe} with $\L=\DQn$ imply that 
\be\label{3.13}
\e_k(\Tr(\Delta Q_n)_\infty, L_1\cap \Tr(\Delta Q_n)) \gg n^{-1/2}.
\ee
   Proposition \ref{Corollary 31.1} and inequality (\ref{3.13}) imply
\be\label{3.14}
\e_k(\Tr(\Delta Q_n)_\infty, L_1) \gg n^{-1/2}.
\ee
Taking into account (\ref{3.12}) and the fact $k\asymp 2^nn^{d-1}$ we derive from (\ref{3.14}) the required lower bound for the $\e_k(\bW^{a,b}_\infty,L_1)$.

The lower bounds in Theorem \ref{T1.5} are proved.

{\bf Proof of lower bounds in Theorem \ref{T1.6}.} The required lower bounds follow from Theorem \ref{T1.5}.

\section{Upper bounds. A general scheme}

{\bf From finite dimensional to infinite dimensional.} Let $X$ and $Y$ be two Banach spaces. We discuss a problem of estimating the entropy numbers of an approximation class, defined in the space $X$, in the norm of the space $Y$. Suppose a sequence of finite dimensional subspaces $X_n \subset X$, $n=1,\dots $, is given. Define the following class 
$$
{\bar \bW}^{a,b}_X:={\bar \bW}^{a,b}_X\{X_n\} := \{f\in X: f=\sum_{n=1}^\infty f_n,\quad  f_n\in X_n, 
$$
$$
 \|f_n\|_X \le 2^{-an}n^{b},\quad n=1,2,\dots\}.
$$
In particular,
$$
 {\bar \bW}^{a,b}_q = {\bar \bW}^{a,b(d-1)}_{L_q}\{\Tr(Q_n)\} .
$$ 
Denote $D_n:=\dim X_n$ and assume that for the unit balls $B(X_n):=\{f\in X_n: \|f\|_X\le 1\}$ we have the following upper bounds for the entropy numbers: there exist real $\al$ and nonnegative   $\ga$ and $\bt\in(0,1]$ such that 
\be\label{EA}
\e_k(B(X_n),Y) \ll n^\al \left\{\begin{array}{ll} (D_n/(k+1))^\bt (\log (4D_n/(k+1)))^\ga, &\quad k\le 2D_n,\\
2^{-k/(2D_n)},&\quad k\ge 2D_n.\end{array} \right.
\ee

\begin{Theorem}\label{T2.1G} Assume $D_n \asymp 2^n n^c$, $c\ge 0$, $a>\bt$, and subspaces $\{X_n\}$ satisfy (\ref{EA}). Then
\be\label{2.0G}
\e_k(\bar \bW^{a,b}_X\{X_n\},Y) \ll k^{-a} (\log k)^{ac+b+\al}.
\ee
\end{Theorem}
\begin{proof} For a given $k$ let $n$ be such that $k\asymp D_n \asymp 2^nn^{c}$. It follows from the definition of class $\bar \bW^{a,b}_X$ that
$$
\e_k(\bar \bW^{a,b}_X,Y) \le \sum_{l=1}^\infty 2^{-al}l^{b}\e_{k_l}(B(X_l),Y),
$$
provided $\sum_{l=1}^\infty k_l \le k$. For $l< n$ we define $k_l := [3a(n-l)D_l/\bt]$. Then 
$\sum_{l=1}^{n-1} k_l \ll k$ and by our assumption (\ref{EA})
$$
\sum_{l=1}^{n-1} 2^{-al} l^{b}\e_{k_l}(B(X_l),Y) 
$$
$$
\ll \sum_{l=1}^n 2^{-al}  l^{b+\al}2^{-k_l/(2D_l)}
$$
$$
\ll 2^{-an}n^{b+\al} \ll k^{-a}(\log k)^{ac+b+\al}.
$$
For $l\ge n$ we define $k_l := [D_n2^{\mu(n-l)}]$, $\mu:=(a-\bt)/(2\bt)$. Then $\sum_{l\ge n} k_l \ll k$.   Therefore, by our assumption (\ref{EA}) we get
$$
\sum_{l\ge n} 2^{-al}  l^{b}\e_{k_l}(B(X_l),Y) 
$$
$$
\ll \sum_{l\ge n} 2^{-al} l^{b+\al}2^{\mu(l-n)\bt} (D_l/D_n)^\bt (l-n)^\ga 
$$
$$ 
\ll 2^{-an}n^{b+\al} \ll k^{-a}(\log k)^{ac+b+\al}.
$$
Thus we proved 
\be\label{2.1G}
\e_{Ck}(\bar \bW^{a,b}_X,Y) \ll k^{-a}(\log k)^{ac+b+\al}.
\ee
Taking into account that the right hand side in (\ref{2.1G}) decays polynomially, we conclude that the upper bound in (\ref{2.0G}) holds. 

\end{proof}

\begin{Remark}\label{R=} In the case $Y=X$ Theorem \ref{T2.1G} holds without assumption (\ref{EA}). It is sufficient to use Corollary \ref{Corollary 32.1}.
\end{Remark}

{\bf From $m$-term approximation to entropy numbers.}
We discuss a technique, which is based on the following two steps strategy. At the first step we obtain bounds of the best $m$-term approximations with respect to a dictionary. At the second step we use general inequalities relating the entropy numbers to the best $m$-term approximations. We begin the detailed discussion with the second step of the above strategy.
Let $\D=\{g_j\}_{j=1}^N$ be a system of elements of cardinality $|\D|=N$ in a Banach space $X$. Consider best $m$-term approximations of $f$ with respect to $\D$
$$
\sigma_m(f,\D)_X:= \inf_{\{c_j\};\Lambda:|\Lambda|=m}\|f-\sum_{j\in \Lambda}c_jg_j\|.
$$
For a function class $F$ set
$$
\sigma_m(F,\D)_X:=\sup_{f\in F}\sigma_m(f,\D)_X.
$$
The following results are from \cite{VT138}.
\begin{Theorem}\label{T3.1} Let a compact $F\subset X$ be such that there exists a   system $\D$, $|\D|=N$, and  a number $r>0$ such that 
$$
  \sigma_m(F,\D)_X \le m^{-r},\quad m\le N.
$$
Then for $k\le N$
\begin{equation}\label{3.0}
\e_k(F,X) \le C(r) \left(\frac{\log(2N/k)}{k}\right)^r.
\end{equation}
\end{Theorem}
\begin{Remark}\label{R3.1} Suppose that a compact $F$ from Theorem \ref{T3.1} belongs to an $N$-dimensional subspace $X_N:=\sp(\D)$. Then in addition to (\ref{3.0}) we have 
 for $k\ge N$
\begin{equation}\label{3.0'}
\e_k(F,X) \le C(r)N^{-r}2^{-k/(2N)}.
\end{equation}
\end{Remark}
We point out that Remark \ref{R3.1} is formulated for a complex Banach space $X$. In the case of real Banach space $X$ we have $2^{-k/N}$ instead of $2^{-k/(2N)}$ in (\ref{3.0'}).

\section{Hyperbolic cross polynomials, $1<q<p<\infty$}
We now proceed to the first step of the above described strategy. First, we discuss 
the entropy numbers $\e_k(\Tr(Q_n)_q,L_p)$ in the case $1<q\le 2 \le p<\infty$. 
For the $m$-term approximation we use the following system described and studied in \cite{VT69}. We define a system of orthogonal trigonometric polynomials which is optimal in a certain sense (see \cite{VT69}) for $m$-term approximations. Variants of this system are well-known and very useful in interpolation of functions by trigonometric polynomials. We begin with a construction of the system $\U$ in the univariate case. Denote
$$
U^+_n(x) := \sum_{k=0}^{2^n-1}e^{ikx} = \frac{e^{i2^nx}-1}{e^{ix}-1},\quad
n=0,1,2,\dots;
$$
$$
U^+_{n,j}(x) := e^{i2^nx}U^+_n(x-2\pi j2^{-n}),\quad j=0,1,\dots ,2^n-1;
$$
$$
 U^-_{n,j}(x) := e^{-i2^nx}U^+_n(-x+2\pi j2^{-n}),\quad j=0,1,\dots ,2^n-1.
$$
It will be more convenient for us to normalize in $L_2$ the system of functions
 $\{U^+_{n,j},U^-_{n,j}\}$. We write
$$
u^+_{n,j}(x) := 2^{-n/2}U^+_{n,j}(x), \qquad u^-_{n,j}(x) := 2^{-n/2}U^-_{n,j}(x).
$$
For $k=2^n+j$, $n=0,1,2,\dots$, and $j=0,1,\dots ,2^n-1$, define
$$
  u_k(x):= u^+_{n,j}(x), \qquad u_{-k}(x):= u^-_{n,j}(x).
$$
The above formulas define $u_k$ for all $k\in \mathbb Z \setminus \{0\}$. Finally, define $u_0(x) =1$. Set $\U :=\{u_k\}_{k\in \Z}$. In the multivariate case of $\bx=(x_1,\dots,x_d)$ we define the system $\U^d$
as the tensor product of the univariate systems $\U$. Namely, $\U^d :=\{u_\bk(\bx)\}_{\bk\in\Z^d}$, where
$$
u_\bk(\bx):= \prod_{i=1}^d u_{k_i}(x_i), \quad \bk=(k_1,\dots,k_d).
$$
For $s\in \N$ denote
$$
\rho^+(s) := \{k:2^{s-1} \le k<2^s\}, \quad \rho^-(s) := \{-k:2^{s-1} \le k<2^s\}, 
$$
and for $s=0$ denote
$$
 \rho^+(0)=\rho^-(0)=\rho(0) :=\{0\}.
$$
Then for $\ep=+$ or $\ep=-$ we have
$$
\Tr(\rho^\ep(s)) = \sp\{u_k, \quad k\in \rho^\ep(s)\} = \sp\{e^{ikx}, \quad k\in \rho^\ep(s)\}.
$$
In the multivariate case for $\bs=(s_1,\dots,s_d)$ and $\ep=(\ep_1,\dots,\ep_d)$ denote
$$
\rho^\ep(\bs):= \rho^{\ep_1}(s_1)\times\cdots\times\rho^{\ep_d}(s_d).
$$
Then
$$
\Tr(\rho^\ep(\bs)) = \sp\{u_\bk, \quad \bk\in \rho^\ep(\bs)\} = \sp\{e^{i(\bk,\bx)}, \quad \bk\in \rho^\ep(\bs)\}.
$$

It is easy to check that for any $\bk \neq \bm$ we have
$$
\<u_\bk,u_\bm\> = (2\pi)^{-d}\int_{\T^d}u_\bk(\bx) \bar u_\bm(\bx) d\bx =0,
$$
and
$$
\|u_\bk\|_2 =1.
$$
We use the notations for $f\in L_1$
$$
f_\bk := \<f,u_\bk\>:=(2\pi)^{-d}\int_{\T^d}f(\bx) \bar u_\bk(\bx) d\bx ;
\qquad \hat f(\bk):= (2\pi)^{-d}\int_{\T^d}f(\bx)e^{-i(\bk,\bx)} d\bx 
$$
and
$$
\de^\ep_\bs(f):= \sum_{\bk\in\rho^\ep(\bs)}\hat f(\bk)e^{i(\bk,\bx)}.
$$
The following important for us analog of the Marcinkiewicz theorem holds
\be\label{tag1.17}
\|\de^\ep_\bs(f)\|^p_p \asymp \sum_{\bk\in \rho^\ep(\bs)}\|f_\bk u_\bk\|^p_p,\quad 1<p<\infty,
\ee
with constants depending on $p$ and $d$.   

We will often use the following inequalities
\be\label{tag1.18}
\left(\sum_{\bs,\ep}\|\de_\bs^\ep(f)\|^p_p\right)^{1/p} \ll \|f\|_p \ll
\left(\sum_{\bs,\ep}\|\de_\bs^\ep(f)\|^2_p\right)^{1/2}, \quad 2\le p < \infty, 
\ee
\be\label{tag1.19}
\left(\sum_{\bs,\ep}\|\de_\bs^\ep(f)\|^2_p\right)^{1/2} \ll \|f\|_p \ll \left(\sum_{\bs,\ep}\|\de_\bs^\ep(f)\|^p_p\right)^{1/p}, 1<p \le 2,
\ee
which are corollaries of the well-known Littlewood-Paley inequalities
\be\label{tag1.20}
\|f\|_p \asymp \|(\sum_\bs|\sum_\ep\de_\bs^\ep(f)|^2)^{1/2}\|_p,\quad 1<p<\infty . 
\ee

\begin{Lemma}\label{L4.1} Let $1<q\le 2\le p<\infty$. Let $\D_n^1 := \{u_\bk:\bk \in \Delta Q_n\}$. Then
\be\label{4.9}
\sigma_m(\Tr(\Delta Q_n)_q,\D_n^1)_p \ll (|\Delta Q_n|/m)^\bt,\qquad \bt =1/q-1/p.
\ee
\end{Lemma}
\begin{proof}   
Theorem \ref{NI} implies Lemma \ref{L4.1} for $m\le m_n := [|Q_n|2^{-n}]$. Let $m\ge m_n$.
Take $f\in \Tr(\Delta Q_n)$.
Then 
$$
f= \sum_{\|\bs\|_1=n} \sum_\ep\delta_\bs^\ep(f).
$$
Represent
$$
\delta_\bs^\ep(f) = \sum_{\bk\in \rho^\ep(\bs)} f_\bk u_\bk.
$$
For convenience we will omit $\ep$ in the notations $\delta_\bs^\ep(f)$, $\rho^\ep(\bs)$,
  meaning that
we are estimating a quantity $\de_\bs^\ep(f)$  for a fixed $\ep$ and all estimates we are going to do in this paper are the same for all $\ep$.   
We now need the following well known simple lemma (see, for instance, \cite{Tmon}, p.92).
\begin{Lemma}\label{Lqp} Let $a_1\ge a_2\ge \cdots \ge a_M\ge 0$ and $1\le q\le p\le \infty$.
Then for all $m<M$ one has
$$
\left(\sum_{k=m}^M a_k^p\right)^{1/p} \le m^{-\bt} \left(\sum_{k=1}^M a_k^q\right)^{1/q},\quad \bt:= 1/q-1/p.
$$
\end{Lemma}
We apply this lemma to each set of $f_\bk$, $\bk\in \rho(\bs)$, $\|\bs\|_1=n$ with $m_\bs := [m/m_n]$. Denote $G_\bs$ the set of cardinality $|G_\bs|=m_\bs$ of $\bk$ from $\rho(\bs)$ with largest $|f_\bk|$. Then by (\ref{tag1.17}) we obtain
$$
\|\sum_{\bk\in \rho(\bs)\setminus G_\bs} f_\bk u_\bk \|_p \asymp 2^{n(1/2-1/p)} \left(\sum_{\bk\in \rho(\bs)\setminus G_\bs} |f_\bk|^p\right)^{1/p}.
$$
Applying Lemma \ref{Lqp} we continue
$$
\ll 2^{n(1/2-1/p)} (m_\bs+1)^{-\bt} \left(\sum_{\bk\in \rho(\bs) } |f_\bk|^q\right)^{1/q}.
$$
Using (\ref{tag1.17}) again we obtain
\be\label{4.10}
\|\sum_{\bk\in \rho(\bs)\setminus G_\bs} f_\bk u_\bk \|_p \ll (|\Delta Q_n|/m)^\bt \|\de_\bs(f)\|_q.
\ee
Estimating the $\|\cdot\|_p$ from above by (\ref{tag1.18}) and the $\|\cdot\|_q$ from below by (\ref{tag1.19}) we complete the proof of Lemma \ref{L4.1}.
\end{proof}

It is easy to see that Lemma \ref{L4.1} implies the corresponding result for $\Tr(Q_n)$.
\begin{Lemma}\label{L4.1Q} Let $1<q\le 2\le p<\infty$. Let $\D_n' := \{u_\bk:\bk \in   Q_n\}$. Then
\be\label{4.9Q}
\sigma_m(\Tr( Q_n)_q,\D_n')_p \ll (|  Q_n|/m)^\bt,\qquad \bt =1/q-1/p.
\ee
\end{Lemma}

We now apply the second step of the strategy described in  Section 4. 
Theorem \ref{T3.1}, Remark \ref{R3.1} and Lemma \ref{L4.1Q} imply the following lemma.

\begin{Lemma}\label{L4.2} Let $1<q\le 2\le p <\infty$ and $\bt:=1/q-1/p$. Then
\be\label{4.11}
\e_k(\Tr( Q_n)_q,L_p) \ll  \left\{\begin{array}{ll} (| Q_n|/k)^\bt (\log (4| Q_n|/k))^\bt, &\quad k\le 2| Q_n|,\\
2^{-k/(2|  Q_n|)},&\quad k\ge 2| Q_n|.\end{array} \right.
\ee
\end{Lemma}

 We now extend Lemma \ref{L4.2} to the case $1<q<p<\infty$. We will use a decomposition technique. 

\begin{Lemma}\label{L35.3pq} Let $1\le u<q<p\le \infty$. For any $f\in L_q(\Omega)$, $\|f\|_q\le 1$, and any positive numbers $a$, $b$ there exists a representation $f=f_1+f_2$ such that
\be\label{35.3pq}
a\|f_1\|_u \le a^{1-\theta}b^{\theta},\quad b\|f_2\|_p \le a^{1-\theta}b^{\theta},\quad \theta:=(1/u-1/q)(1/u-1/p)^{-1} .
\ee
\end{Lemma} 
\begin{proof} Let $f_T$ denote the $T$ cut off of $f$: $f_T(x)=f(x)$ if $|f(x)|\le T$ and $f_T(x)=0$ otherwise. Clearly, $\|f_T\|_\infty \le T$.  Set $f^T := f-f_T$. 
We now estimate the $L_u$ norm of the $f^T$. Let $E:=\{x:f^T(x)\neq 0\}$. First, we bound from above the measure $|E|$ of $E$. We have
$$
1\ge \int_\Omega |f(\bx)|^qd\bx \ge \int_ET^qd\bx = T^q|E|.
$$
Second, we bound the $\|f^T\|_u$
$$
\|f^T\|_u^u = \int_E|f^T(\bx)|^ud\bx \le \left(\int_E|f^T(\bx)|^qd\bx\right)^{u/q}|E|^{1-u/q} \le T^{u-q}.
$$
Third, we bound the $\|f_T\|_p$. Using the inequality
$$
\|g\|_p \le \|g\|_q^{q/p}\|g\|_\infty^{1-q/p}
$$
we obtain
$$
\|f_T\|_p \le T^{1-q/p}.
$$
Specifying $T=(a/b)^{(q(1/u-1/p))^{-1}}$ we get
$$
a\|f^T\|_u \le a^{1-\theta}b^{\theta},\qquad b\|f_T\|_p \le a^{1-\theta}b^{\theta}.
$$
This proves the lemma.
\end{proof}

We now derive from Lemma \ref{L35.3pq} the following inequality for the entropy numbers.

\begin{Lemma}\label{L2qp} For $1<u<q<p<\infty$ we have for $\theta:=(1/u-1/q)(1/u-1/p)^{-1}$
$$
\e_{k+l}(\Tr(Q_n)_q,L_p) \le C(u,p,d) \left(\e_k(\Tr(Q_n)_u,L_p)\right)^{1-\theta}\left(\e_l(\Tr(Q_n)_p,L_p)\right)^{\theta}.
$$
\end{Lemma}
\begin{proof} Let $t\in \Tr(Q_n)_q$. Applying Lemma \ref{L35.3pq} we split the polynomial $t$ into a sum $t=f_1+f_2$ satisfying (\ref{35.3pq}). Consider 
$$
t_1:= S_{Q_n}(f_1),\qquad t_2:= S_{Q_n}(f_2).
$$
Then
$$
at_1 \in \Tr(Q_n)_u C(u,d) a^{1-\theta}b^{\theta} \quad \text{and}\quad bt_2 \in \Tr(Q_n)_p C(p,d) a^{1-\theta}b^{\theta}.
$$
 Let $a$ and $b$ be such that 
$$
a>\e_k(\Tr(Q_n)_u,L_p),\qquad b>\e_l(\Tr(Q_n)_p,L_p).
$$
Find $y_1,\dots,y_{2^k}$ and $z_1,\dots,z_{2^l}$ such that
$$
\Tr(Q_n)_u\subset \cup_{i=1}^{2^k}B_{L_p}(y_i,a),\qquad \Tr(Q_n)_p\subset \cup_{j=1}^{2^l}B_{L_p}(z_j,b).
$$
Take any $t\in \Tr(Q_n)_q$. Set $\e:=a^{1-\theta}b^{\theta}\max\{C(u,d),C(p,d)\}$ and as above find $t_1$ and $t_2$ from $\Tr(Q_n)$ such that $t=t_1+t_2$ and
$$
a\|t_1\|_u \le  \e,\qquad b\|t_2\|_p \le  \e.
$$
Clearly, for some $i$
\begin{equation}\label{35.10gen}
at_1/\e\in B_{L_p}(y_i,a)\quad \Rightarrow \quad t_1\in B_{L_p}(\e y_i/a,\e)
\end{equation}
and for some $j$
\begin{equation}\label{35.11gen}
bt_2/\e\in B_{L_p}(z_j,b)\quad \Rightarrow \quad t_2\in B_{L_p}(\e z_j/b,\e).
\end{equation}

Consider the sets $G_{i,j}:=B_{L_p}(\e y_i/a+\e z_j/b, 2\e)$, $i=1,\dots, 2^k$, $j=1,\dots, 2^l$. Relations (\ref{35.10gen}) and (\ref{35.11gen}) imply $t\in G_{i,j}$. Thus 
$$
\e_{k+l}(\Tr(Q_n)_q,L_p) \le 2\e.
$$

\end{proof}

Let $1\le q< p\le v\le \infty$. The following simple inequality
$$
\|g\|_p \le \|g_q\|_q^{1-\mu}\|g\|_v^{\mu},\qquad \mu:= (1/q-1/p)(1/q-1/v)^{-1}
$$
implies that
\be\label{qp<2}
\e_k(\Tr(Q_n)_q,L_p) \le 2\left(\e_k(\Tr(Q_n)_q,L_v)\right)^\mu.
\ee

\begin{Lemma}\label{L4.2gen} Let $1<q< p <\infty$ and $\bt:=1/q-1/p$. Then
$$
\e_k(\Tr(Q_n)_q,L_p) \ll  \left\{\begin{array}{ll} (|Q_n|/k)^\bt (\log (4|Q_n|/k))^\bt, &\quad k\le 2|Q_n|,\\
2^{-k/(2|Q_n|)},&\quad k\ge 2| Q_n|.\end{array} \right.
$$
\end{Lemma}
 \begin{proof} In the case $1<q\le 2\le p<\infty$ Lemma \ref{L4.2gen} follows directly from Lemma \ref{L4.2}. Corollary \ref{Corollary 32.1} shows that it is sufficient to prove Lemma \ref{L4.2gen} for $k\le 2|Q_n|$.
 Consider first the case $2<q<p<\infty$. Applying Lemma  \ref{L2qp} with $u=2$, $l=0$ and using Lemma \ref{L4.2} with $q=2$ we obtain the required bound. Second, in the case $1<q<p<2$ the required bound follows from inequality (\ref{qp<2}) with $v=2$ and Lemma \ref{L4.2} with $p=2$.
 \end{proof}
 
 Theorem \ref{T2.1G} and Lemma \ref{L4.2gen} imply the upper bounds in the following theorem. 

\begin{Theorem}\label{T4.2} Let $1<q<p <\infty$ and $a>\bt:=1/q-1/p$. Then
\be\label{4.12}
\e_k(\bW^{a,b}_q,L_p)\asymp \e_k(\bar \bW^{a,b}_q,L_p) \asymp k^{-a} (\log k)^{(d-1)(a+b)}.
\ee
\end{Theorem}

The lower bound in this theorem follows from the lower bound in Theorem   \ref{T1.7} above.

\section{Hyperbolic cross polynomials, $d=2$, $q=1$ and $q=\infty$}

The construction of the orthonormal basis from the previous section uses classical
methods and classical building blocks -- an analog of the Dirichlet kernels. That construction works very well for $L_q$ spaces with $q\in(1,\infty)$. However, because of use of the Dirichlet kernels it does not work well in the case $q=1$ or $q=\infty$. 
We present here other construction, which is based on the wavelet theory. This construction is taken from \cite{OO}. Let $\de$ be a fixed number, $0<\de\le1/3$, and let $\hat \ff(\la) = \hat \ff_\de(\la)$, $\la\in \R$, be a sufficiently smooth function (for simplicity, real-valued and even) equal $1$ for $|\la| \le (1-\de)/2$, equal to $0$ for $|\la|>(1+\de)/2$ and such that the integral translates of its square constitute a partition of unity:
\be\label{O1}
\sum_{k\in\Z}(\hat\ff(\la+k))^2 =1,\qquad \la\in\R.
\ee
It is known that condition (\ref{O1}) is equivalent to the following property: The set of 
functions $\Phi:=\{\ff(\cdot+l)\}_{l\in\Z}$, where
$$
\ff(x) = \int_\R\hat\ff(\la)e^{2\pi i\la x}d\la,
$$
is an orthonormal system on $\R$:
\be\label{O1'}
\int_\R\ff(x+k)\ff(x+l)dx =\de_{k,l},\qquad k,l\in\Z.
\ee
Following \cite{OO} define
$$
\theta(\la) := \left(((\hat\ff(\la/2))^2-(\hat\ff(\la))^2\right)^{1/2}
$$
and consider, for $n=0,1,\dots$, the trigonometric polynomials
\be\label{O3}
\Psi_n(x):=2^{-n/2}\sum_{k\in\Z}\theta(k2^{-n})e^{2\pi ikx}.
\ee
Introduce also the following dyadic translates of $\Psi_n$:
$$
\Psi_{n,j}(x):= \Psi_n(x-(j+1/2)2^{-n}),
$$
and define the sequence of polynomials $\{T_k\}_{k=0}^\infty$
\be\label{O4}
T_0(x):=1,\qquad T_k(x):= \Psi_{n,j}(x)  
\ee
if $k=2^n+j$, $n=0,1,\dots$, $0\le j<2^n$. Note that $T_k$ is the trigonometric polynomial such that
\be\label{O4'}
\hat T_k(\nu) =0\quad\text{if} \quad |\nu|\ge 2^n(1+\de)\quad \text{or}\quad 
|\nu|\le 2^{n-1}(1-\de).
\ee

It is proved in \cite{OO} that the system $\{T_k\}_{k=0}^\infty$ is a complete orthonormal basis in all $L_p$, $1\le p\le \infty$ (here, $L_\infty$ stands for the space of continuous functions) of $1$-periodic functions. Also, it is proved in \cite{OO} that
\be\label{O12}
|\Psi_n(x)| \le C(\kappa,\de)2^{n/2}(2^n|\sin \pi x|+1)^{-\kappa}
\ee
with $\kappa$ determined by the smoothness of $\hat\ff(\la)$. In particular, we can always make $\kappa >1$ assuming that $\hat\ff(\la)$ is smooth enough. 
It is more convenient for us to consider $2\pi$-periodic functions. We define $\V:=\{v_k\}_{k=0}^\infty$ with 
$v_k(x):=T_k(x/(2\pi))$ for $x\in [0,2\pi)$. 

In the multivariate case of $\bx=(x_1,\dots,x_d)$ we define the system $\V^d$
as the tensor product of the univariate systems $\V$. Namely, $\V^d :=\{v_\bk(\bx)\}_{\bk\in\Z_+^d}$, where
$$
v_\bk(\bx):= \prod_{i=1}^d v_{k_i}(x_i), \quad \bk=(k_1,\dots,k_d).
$$
 
 Property (\ref{O12}) implies the following simple lemma.
 \begin{Lemma}\label{LO1} We have
 $$
 \|\sum_{\bk\in\rho^+(\bs)} a_\bk v_\bk\|_\infty \le C(d,\kappa,\de)2^{\|\bs\|_1/2} \max_\bk|a_\bk|.
 $$
 \end{Lemma}
 
 We use the notation
 $$
 f_\bk:= \<f,v_\bk\> = (2\pi)^{-d} \int_{\T^d} f(\bx)v_\bk(\bx)d\bx.
 $$
 Denote
 $$
 Q_n^+ := \{ \bk =(k_1,\dots,k_d)\in Q_n : k_i\ge 0, i=1,2,\dots\},\qquad \theta_n:=\{\bs: \|\bs\|_1=n\},
 $$
 $$
 \V(Q_n) := \{f: f=\sum_{\bk\in Q_n^+} c_\bk v_\bk\},\quad \V(Q_n)_A := \{f\in\V(Q_n): \sum_{\bk\in Q_n^+} |f_\bk| \le 1\}. 
 $$
 We prove three inequalities for $f\in \V(Q_n)$ in the case $d=2$. Theorem \ref{T1v} is a generalized version of the Small Ball Inequality for the system $\V^2$.
 
 \begin{Theorem}\label{T1v} Let $d=2$. For any $f\in \V(Q_n)$ we have for $l\le n$
 $$
 \sum_{\bs\in\theta_l} \|\sum_{\bk\in \rho^+(\bs)} f_\bk v_\bk \|_1 \le C(6+n-l)\|f\|_\infty,
 $$
 where the constant $C$ may depend on the choice of $\hat\ff$.
 \end{Theorem} 
 
   \begin{Theorem}\label{T1lv} Let $d=2$. For any $f\in \V(Q_n)$ we have
 $$
 \sum_{\bk\in Q_n^+} |f_\bk|  \le C2^{n/2}\|f\|_\infty,
 $$
 where the constant $C$ may depend on the choice of $\hat\ff$.
 \end{Theorem} 
 
  \begin{Theorem}\label{T2v} Let $d=2$. For any $f\in \V(Q_n)$ we have
 $$
 \sum_{\bk\in Q_n^+} |f_\bk|  \le C|Q_n|^{1/2}\|f\|_1,
 $$
 where the constant $C$ may depend on the choice of $\hat\ff$.
 \end{Theorem} 

Proofs of Theorems \ref{T1v} and \ref{T2v} are based on the Riesz products for the hyperbolic cross polynomials (see \cite{Tem3}, \cite{Tmon}, \cite{TE3}, \cite{TE4}). Relation (\ref{O4'}) implies that for $s_1$ and $s_2$ greater than $3$ the function $v_\bk$ with $\bk\in \rho^+(\bs)$ may have nonzero Fourier coefficients $\hat v_\bk(\bm)$ only for 
$$
\bm \in \rho'(\bs) := \{\bm=(m_1,m_2): (1-\de)2^{s_i-2}<|m_i|<(1+\de)2^{s_i-1}, i=1,2\}.
$$
In other words
$$
v_\bk\in \Tr(\rho'(\bs)),\qquad \bk \in \rho^+(\bs).
$$
We introduce some more notations. For any two integers $a\ge 1$ and $0\le b<a$, we shall denote by $AP(a,b)$ the arithmetical progression of the form $al+b$, $l=0,1,\dots$. Set
$$
H_n(a,b) := \{\bs=(s_1,s_2): \bs\in \Z_+^2,\quad \|\bs\|_1=n, \quad s_1,s_2\ge a,\quad s_1\in AP(a,b)\}.
$$
For a subspace $Y$ in $L_2(\T^d)$ we denote by $Y^\perp$ its orthogonal complement. 
We need the following lemma on the Riesz product, which is a variant of Lemma 2.1 from \cite{TE3}. 

\begin{Lemma}\label{L2.1v} Take any trigonometric polynomials $t_\bs\in \Tr(\rho'(\bs))$ and form the function
$$
\Phi(\bx) := \prod_{\bs\in H_n(a,b)}(1+t_\bs).
$$
Then for any $a\ge 6$ and any $0\le b<a$ this function admits the representation
$$
\Phi(\bx) = 1+ \sum_{\bs\in H_n(a,b)} t_\bs(\bx) +g(\bx)
$$
with $g\in \Tr(Q_{n+a-6})^\perp$.
\end{Lemma}
\begin{proof} We prove that for $\bk=(k_1,k_2)$ such that $|k_1k_2|\le 2^{n+a-6}$ we have $\hat g(\bk)=0$. This proof follows the ideas from \cite{Tem3}. Let $w(kt)$ denote either $\cos kt$ or $\sin kt$. Then $g(\bx)$ contains terms of the form
\be\label{2.1v}
h(\bx) = c\prod_{i=1}^m w(k^i_1x_1)w(k^i_2x_2), \quad \bk^i \in \rho'(\bs^i), 
\ee
with all $\bs^i$, $i=1,\dots,m$, $m\ge 2$, distinct. For the sake of simplicity of notations we assume that 
$s^1_1>s^2_1>\cdots >s^m_1$. Then for $h(\bx)$ the frequencies with respect to $x_1$ have the form
$$ 
k_1 = k^1_1 \pm k^2_1 \pm \cdots \pm k^m_1.
$$
Therefore, for $\de\le 1/3$ and $a\ge 6$ we obtain
\be\label{2.2v}
k_1 > (1-\de)2^{s^1_1-2} - \sum_{i\ge 1} (1+\de)2^{s^1_1-1-ai} > 2^{s^1_1-3}.
\ee
In the same way for frequencies of the function $h(\bx)$ with respect to $x_2$ we get
$k_2> 2^{s^m_2-3}$. Consequently, 
$$
k_1k_2 > 2^{s^1_1+s^m_2-6}.
$$
In order to complete the proof it remains to observe that for all terms $h(\bx)$ of the function $g(\bx)$ we have $m\ge 2$, which implies $s^1_1+s^m_2 \ge n+a$. The lemma is proved.

\end{proof}

{\bf Proof of Theorem \ref{T1v}.} For a rectangle $R \subset \Z_+^2$ denote
$$
S_R(f,\V):= \sum_{\bk\in R} f_\bk v_\bk,\qquad \de_\bs(f,\V) := S_{\rho^+(\bs)}(f,\V). 
$$
It is proved in \cite{OO} that $\|S_R\|_{L_\infty \to L_\infty} \le B$, where $B$ may depend only on the function $\hat\ff$. Define
$$
t_\bs := S_{\rho^+(\bs)}(\sign \de_\bs(f,\V))B^{-1}.
$$
Then $t_\bs\in \Tr(\rho'(\bs))$ and $\|t_\bs\|_\infty \le 1$. By Lemma \ref{L2.1v} with $n$ replaced by $l$ and $a=6+n-l$, where $n$ is from Theorem \ref{T1v} we obtain 
$$
\Phi(\bx) = 1+ \sum_{\bs\in H_l(6+n-l,b)} t_\bs(\bx) +g(\bx)
$$
with $g\in \Tr(Q_{n})^\perp$. Clearly, $\|\Phi\|_1 =1$. Therefore, on one hand 
\be\label{1v}
\<f,\Phi-1\> \le 2\|f\|_\infty.
\ee
On the other hand 
$$
\<f,\Phi-1\> = \sum_{\bs\in H_l(6+n-l,b)} \<f,t_\bs\> = \sum_{\bs\in H_l(6+n-l,b)} \<\de_\bs(f,\V),t_\bs\>
$$
\be\label{2v}
= \sum_{\bs\in H_l(6+n-l,b)} \<\de_\bs(f,\V),\sign \de_\bs(f,\V)\>B^{-1} = B^{-1}\sum_{\bs\in H_l(6+n-l,b)} \|\de_\bs(f,\V)\|_1.
\ee
Thus, for each $0\le b<6+n-l$ we have
\be\label{3v}
\sum_{\bs\in H_l(6+n-l,b)} \|\de_\bs(f,\V)\|_1 \le 2B\|f\|_\infty
\ee
This easily implies the conclusion of Theorem \ref{T1v}.

{\bf Proof of Theorem \ref{T1lv}.} Theorem \ref{T1lv} is a corollary of Theorem \ref{T1v}. 
Indeed, by Lemma \ref{LO1} we get
$$
\sum_{\bk\in\rho^+(\bs)} |f_\bk| = \<\de_\bs(f,\V), \sum_{\bk\in\rho^+(\bs)} (\sign f_\bk)v_\bk\> \le C\|\de_\bs(f,\V)\|_1 2^{\|\bs\|_1/2}.
$$
Thus, by Theorem \ref{T1v} we get
$$
\sum_{\bk\in Q^+_n} |f_\bk| \le C\sum_{l\le n}(6+n-l)2^{l/2}\|f\|_\infty \ll 2^{n/2}\|f\|_\infty,
$$
which completes the proof of Theorem \ref{T1lv}.

{\bf Proof of Theorem \ref{T2v}.} We begin with some auxiliary results. The following simple remark is from \cite{TE4}.

\begin{Remark}\label{R2.1v} For any real numbers $y_l$ such that $|y_l|\le 1$, $l=1,\dots,N$ we have ($i^2 =-1$)
$$
\left|\prod_{l=1}^N \left(1+\frac{iy_l}{\sqrt{N}}\right)\right| \le C.
$$
\end{Remark}

We now prove two lemmas, which are analogs of Lemmas 2.2 and 2.3 from \cite{TE4}. Denote
$$
E^\perp_{Q_n}(f)_p := \inf_{g\in \Tr(Q_n)^\perp} \|f-g\|_p.
$$
\begin{Lemma}\label{L2.2v} For any function $f$ of the form
$$
f=\sum_{\bs\in H_n(a,b)} t_\bs
$$
with $a\ge 6$, $0\le b<a$, where $t_\bs$, $\bs\in H_n(a,b)$, is a real trigonometric polynomial in $\Tr(\rho'(\bs))$ such that $\|t_\bs\|_\infty\le 1$, we have
$$
E^\perp_{Q_{n+a-6}}(f)_\infty \le C(1+n/a)^{1/2}
$$
with $C$ depending only on $\hat\ff$.
\end{Lemma}
\begin{proof} Let us form the function
$$
RP(f) := \text{Im} \prod_{\bs\in H_n(a,b)} \left(1+it_\bs(1+n/a)^{-1/2}\right),
$$
which is an analog of the Riesz product. Then by Remark \ref{R2.1v} we have
\be\label{4v}
\|RP(f)\|_\infty \le C.
\ee
Lemma \ref{L2.1v} provides the representation 
\be\label{5v}
RP(f) = (1+n/a)^{-1/2} \sum_{\bs\in H_n(a,b)} t_\bs + g,\quad g\in \Tr(Q_{n+a-6}).
\ee
Combining (\ref{4v}) and (\ref{5v}) we obtain the statement of Lemma \ref{L2.2v}.
\end{proof}

\begin{Remark}\label{R2.2v} It is clear that in Lemma \ref{L2.2v} we can drop the assumption that the $t_\bs$ are real polynomials.
\end{Remark}

\begin{Lemma}\label{L2.3v} For any function $f$ of the form 
$$
f=\sum_{\bs\in \theta_n} t_\bs,\quad t_\bs\in \Tr(\rho'(\bs)),\quad \|t_\bs\|_\infty\le 1,
$$
we have for any $a\ge 6$
$$
E^\perp_{Q_{n+a-6}}(f)_\infty \le Ca(1+n/a)^{1/2}
$$
\end{Lemma}
\begin{proof} Let us introduce some more notations. Denote
$$
\theta_{n,a} := \{\bs\in\theta_n:\quad\text{either} \quad s_1<a \quad \text{or}\quad s_2<a\}.
$$
Then
$$
f=\sum_{\bs\in\theta_n} t_\bs = \sum_{\bs\in\theta_{n,a}} t_\bs + \sum_{b=0}^{a-1}\sum_{\bs\in H_n(a,b)} t_\bs
$$
and
$$
E^\perp_{Q_{n+a-6}}(f)_\infty \le \sum_{\bs\in\theta_{n,a}} \|t_\bs\|_\infty + \sum_{b=0}^{a-1}E^\perp_{Q_{n+a-6}}\left(\sum_{\bs\in H_n(a,b)} t_\bs\right)_\infty.
$$
Using the assumption $\|t_\bs\|_\infty\le 1$, Lemma \ref{L2.2v}, and Remark \ref{R2.2v} we get from here the required estimate.
\end{proof}

We now proceed to the proof of Theorem \ref{T2v}. For $l\in [0,n]$ consider
$$
t^1_\bs := \sum_{\bk\in\rho^+(\bs)}(\sign f_\bk) v_\bk,\quad M_l:=\max_{\bs\in\theta_l}\|t^1_\bs\|_\infty,\quad t_\bs := t^1_\bs/M_l.
$$
By Lemma \ref{LO1}
$$
M_l \ll 2^{l/2}.
$$
Applying Lemma \ref{L2.3v} with $a=6+n-l$ we get
$$
\sum_{\bs\in\theta_l} \sum_{\bk\in\rho^+(\bs)} |f_\bk| = \<f,\sum_{\bs\in\theta_l} t^1_\bs\>= M_l \<f,\sum_{\bs\in\theta_l} t_\bs\>
$$
$$
\ll 2^{l/2} E^\perp_{Q_n}\left(\sum_{\bs\in\theta_l} t_\bs\right)_\infty \|f\|_1 \ll 2^{l/2} (6+n-l)\left(1+\frac{n}{6+n-l}\right)^{1/2}.
$$
Summing up over $l\le n$, we complete the proof of Theorem \ref{T2v}.

\begin{Lemma}\label{L4.1v} Let $2\le p<\infty$. Let $\V_n^1 := \{v_\bk:\bk \in Q_n\}$. Then
\be\label{4.9v}
\sigma_m(\V(Q_n)_A,\V_n^1)_p \ll |Q_n|^{1/2-1/p}m^{1/p-1}.
\ee
\end{Lemma}
\begin{proof} Note, that for $f\in\V(Q_n)_A$ we easily obtain that $\|f\|_2\le 1$ and $\|f\|_\infty \ll 2^{n/2}$, which, in turn, implies for $2\le p\le\infty$
\be\label{6.15}
\|f\|_p   \ll 2^{n(1/2-1/p)}.
\ee
Thus, it is sufficient to prove (\ref{4.9v}) for big enough $m$. 

First, we prove the lemma for $\DQn$ instead of $Q_n$. Then $f\in \V(\DQn)_A$ has a representation
$$
f=\sum_{\bk\in \DQn} f_\bk v_\bk, \qquad \sum_{\bk\in \DQn} |f_\bk| \le 1.
$$
Using the fact that the system $\V^d$ is orthonormal we obtain by Lemma \ref{Lqp} 
with $m_1:=[m/2]$ that 
\be\label{4.30}
\sigma_{m_1}(f,\V^d)_2 \le (m_1+1)^{-1/2}.
\ee
For a set $\L$ denote
$$
\V(\L)_q := \{f: f=\sum_{\bk\in \L\cap \Z^d_+}f_\bk v_\bk, \quad \|f\|_q\le 1\}.
$$
Next, we estimate the best $m_1$-term approximation of $g\in \V(\DQn)_2$ in $L_p$, $2<p<\infty$.  We apply Lemma \ref{Lqp} to each set of $g_\bk$, $\bk\in \rho^+(\bs)$, $\|\bs\|_1=n$ with $m_\bs := [m_1/m_n]$, $m_n := [|\DQn|2^{-n}]$, assuming that $n\ge C$ with the absolute constant $C$ large enough to guarantee $m_n\ge 1$. Denote $G_\bs$ the set of cardinality $|G_\bs|=m_\bs$ of $\bk$ from $\rho^+(\bs)$ with largest $|g_\bk|$. Then by Lemmas \ref{Lqp} and \ref{LO1} we obtain
$$
\|\sum_{\bk\in \rho^+(\bs)\setminus G_\bs} g_\bk v_\bk \|_\infty \ll 2^{n/2} (m_\bs+1)^{-1/2}\|\de_\bs(g,\V)\|_2.
$$
Applying simple inequality for $2\le p\le \infty$
$$
\|f\|_p\le \|f\|_2^\al \|f\|_\infty^{1-\al},\qquad \al =2/p
$$
  we obtain
\be\label{4.10v}
\|\sum_{\bk\in \rho(\bs)\setminus G_\bs} g_\bk v_\bk \|_p \ll (|\Delta Q_n|/m_1)^{1/2-1/p} \|\de_\bs(g,\V)\|_2.
\ee
Inequality (\ref{tag1.18}) implies easily a similar inequality for $\V^d$ for $2\le p<\infty$
\be\label{4.32}
\|f\|_p \ll \left(\sum_\bs\|\de_\bs(f,\V)\|_p^2\right)^{1/2}.
\ee
Combining (\ref{4.30}), (\ref{4.10v}), and (\ref{4.32}) we
 complete the proof of Lemma \ref{L4.1v} in the case of $\DQn$.

 We now derive the general case of $Q_n$ from the above considered case of $\DQl$. Set 
 $$
 \mu := \frac{1}{2}\left(\frac{1}{2}-\frac{1}{p}\right)\left(1-\frac{1}{p}\right)^{-1}
 $$
 and denote by $l_0$ the smallest $l$ satisfying 
 \be\label{6.19}
 m2^{-\mu(n-l)} \ge 1.
 \ee
 Then for $f\in \V(Q_n)_A$ we have by (\ref{6.15}) and (\ref{6.19})
 \be\label{6.20}
 f_0:=\sum_{\bk\in Q^+_{l_0}}f_\bk v_\bk, \quad \|f_0\|_p \ll 2^{l_0(1/2-1/p)}\ll 2^{n(1/2-1/p)}m^{2(1/p-1)}.
 \ee 
  For $l>l_0$ define $m_l:=[m2^{-\mu(n-l)}] \ge 1$. Then
  $$
  m':= \sum_{l_0<l\le n} m_l \le C(p) m,
  $$
  and 
\be\label{6.21}
\sigma_{m'}(f-f_0, \V^1_n)_p \ll \sum_{l_0<l\le n} |\DQl|^{1/2-1/p}m_l^{1/p-1} \ll   |Q_n|^{1/2-1/p}m^{1/p-1}.
\ee
Combining (\ref{6.20}) -- (\ref{6.21}) we complete the proof of Lemma \ref{L4.1v}. 
  
\end{proof}

Lemma \ref{L4.1v} and Theorem \ref{T2v} imply.
\begin{Lemma}\label{L4.1v1} Let $2\le p<\infty$. Let $\V_n^1 := \{v_\bk:\bk \in Q_n\}$. Then
$$
\sigma_m(\V(Q_n)_1,\V_n^1)_p \ll (|Q_n|/m)^{1-1/p}.
$$
\end{Lemma}

Lemma \ref{L4.1v} and Theorem \ref{T1lv} imply.
\begin{Lemma}\label{L4.1v2} Let $2\le p<\infty$. Let $\V_n^1 := \{v_\bk:\bk \in Q_n\}$. Then
$$
\sigma_m(\V(Q_n)_\infty,\V_n^1)_p \ll n^{-1/2}(|Q_n|/m)^{1-1/p}.
$$
\end{Lemma}

We now apply the second step of the strategy described in  Section 4. 
Theorem \ref{T3.1}, Remark \ref{R3.1} and Lemma \ref{L4.1v1} imply the following lemma.

\begin{Lemma}\label{L4.2v} Let $ 2\le p <\infty$ and $\bt:=1-1/p$. Then
$$
\e_k(\V(Q_n)_1,L_p) \ll  \left\{\begin{array}{ll}  (|Q_n|/k)^\bt (\log (4|Q_n|/k))^\bt, &\quad k\le 2|Q_n|,\\
 2^{-k/(2|Q_n|)},&\quad k\ge 2|Q_n|.\end{array} \right.
$$
\end{Lemma}
Theorem \ref{T3.1}, Remark \ref{R3.1} and Lemma \ref{L4.1v2} imply the following lemma.

\begin{Lemma}\label{L4.2vc} Let $ 2\le p <\infty$ and $\bt:=1-1/p$. Then
$$
\e_k(\V(Q_n)_\infty,L_p) \ll  \left\{\begin{array}{ll} n^{-1/2}(|Q_n|/k)^\bt (\log (4|Q_n|/k))^\bt, &\quad k\le 2|Q_n|,\\
n^{-1/2}2^{-k/(2|Q_n|)},&\quad k\ge 2|Q_n|.\end{array} \right.
$$
\end{Lemma}

The fact that $\Tr(Q_n) \subset \V(Q_{n+6})$, Corollary \ref{Corollary 32.1} and inequality (\ref{qp<2}) allow us to derive the following results from Lemmas \ref{L4.2v} and \ref{L4.2vc}.

\begin{Theorem}\label{T4.2v} Let $ 1< p <\infty$ and $\bt:=1-1/p$. Then
$$
\e_k(\Tr(Q_n)_1,L_p) \ll  \left\{\begin{array}{ll}  (|Q_n|/k)^\bt (\log (4|Q_n|/k))^\bt, &\quad k\le 2|Q_n|,\\
 2^{-k/(2|Q_n|)},&\quad k\ge 2|Q_n|.\end{array} \right.
$$
\end{Theorem}

\begin{Theorem}\label{T4.2vc} Let $ 2\le p <\infty$ and $\bt:=1-1/p$. Then
$$
\e_k(\Tr(Q_n)_\infty,L_p) \ll  \left\{\begin{array}{ll} n^{-1/2}(|Q_n|/k)^\bt (\log (4|Q_n|/k))^\bt, &\quad k\le 2|Q_n|,\\
n^{-1/2}2^{-k/(2|Q_n|)},&\quad k\ge 2|Q_n|.\end{array} \right.
$$
\end{Theorem}

\section{Hyperbolic cross polynomials, $p=\infty$}

We now discuss a more difficult and more interesting case $p=\infty$. Denote 
$$
\|f\|_A := \sum_\bk |{\hat f}(\bk)|,\quad {\hat f}(\bk):= (2\pi)^{-d}\int_{\T^d} f(\bx)e^{-i(\bk,\bx)}d\bx.
$$
The following Theorem \ref{T2.6} is from \cite{VT150} (see Theorem 2.6). We use it in this paper. 
Let as above
$$
\Pi(\bN,d) :=\bigl \{(a_1,\dots,a_d)\in \Z^d  : |a_j|\le N_j,\
j = 1,\dots,d \bigr\} ,
$$
where $N_j$ are nonnegative integers and $\bN:=(N_1,\dots,N_d)$. We denote
$$
\Tr(\bN,d):= \Tr(\Pi(\bN,d))=\{t:t = \sum_{\bk\in \Pi(\bN,d)} c_\bk e^{i(\bk,\bx)}\}.
$$
Then 
$$
\dim \Tr(\bN,d) = \prod_{j=1}^d (2N_j  + 1) =: \vartheta(\bN).
$$
For a nonnegative integer $m$ denote $\mb:= \max(m,1)$.

\begin{Theorem}\label{T2.6} There exist constructive greedy-type approximation methods $G^p_m(\cdot)$, which provide $m$-term polynomials with respect to $\Tr^d$ with the following properties: for $2\le p<\infty$
$$
\|f-G^p_m(f)\|_p \le C_1(d)(\mb)^{-1/2}p^{1/2}\|f\|_A, 
$$
  and for $p=\infty$, $f\in \Tr(\bN,d)$
$$
\|f-G^\infty_m(f)\|_\infty \le C_2(d)(\mb)^{-1/2}(\ln \vartheta(\bN))^{1/2}\|f\|_A.
$$
\end{Theorem}

We will use a version of Theorem \ref{T2.6}, which follows from the proof of Theorem \ref{T2.6} in \cite{VT150} and a simple inequality 
$$
\|f\|_A \le \|f\|_2 (\#\{\bk:\hat f(\bk)\neq 0\})^{1/2}.
$$
\begin{Theorem}\label{T2.6'} Let $\Lambda\subset \Pi(\bN,d)$ with $N_j=2^n$, $j=1,\dots,d$. There exist constructive greedy-type approximation methods $G^\infty_m(\cdot)$, which provide $m$-term polynomials with respect to $\Tr^d$ with the following properties: 
\newline 
for $f\in \Tr(\Lambda)$ we have $G^\infty_m(f)\in \Tr(\Lambda)$ and 
$$
\|f-G^\infty_m(f)\|_\infty \le C_3(d)(\mb)^{-1/2}n^{1/2}|\Lambda|^{1/2}\|f\|_2.
$$
\end{Theorem}

We now prove the following lemma. Let $\D_n^2 := \{u_\bk:\bk\in   Q_n\} \cup \{e^{i(\bk,\bx)}: \bk\in   Q_n\}$.

\begin{Lemma}\label{L4.3} Let $1<q\le 2$. Then
$$
\sigma_m(\Tr(  Q_n)_q,\D_n^2)_\infty \ll n^{1/2}(|  Q_n|/m)^{1/q}.
$$
\end{Lemma}
\begin{proof} 
 Take $f\in \Tr(  Q_n)$. Applying first Lemma \ref{L4.1Q} with $p=2$ and $[m/2]$ and, then, applying 
Theorem \ref{T2.6'} with $\Lambda =   Q_n$ and $[m/2]$ we obtain
$$
\sigma_m(f,\D_n^2)_\infty \ll n^{1/2}(| Q_n|/m)^{1/q}\|f\|_q,
$$
which proves the lemma.
\end{proof}

Theorem \ref{T3.1}, Remark \ref{R3.1} and Lemma \ref{L4.3} imply the following lemma.

\begin{Theorem}\label{T4.4} Let $1<q\le 2$. Then
$$
\e_k(\Tr( Q_n)_q,L_\infty) \ll  \left\{\begin{array}{ll} n^{1/2}(| Q_n|/k)^{1/q} (\log (4| Q_n|/k))^{1/q}, &\quad k\le 2| Q_n|,\\
n^{1/2}2^{-k/(2| Q_n|)},&\quad k\ge 2| Q_n|.\end{array} \right.
$$
\end{Theorem}

Let us discuss the case $q=1$, which is not covered by Theorem \ref{T4.4}. In this case we restrict ourselves to $d=2$. It is easy to see that $\V(Q_n) \subset \Tr(Q_n)$. Let $\V_n^2 := \{v_\bk:\bk\in   Q_n\} \cup \{e^{i(\bk,\bx)}: \bk\in   Q_{n}\}$. Then we have the following analog of Lemma \ref{L4.3}.

\begin{Lemma}\label{L4.3v} We have for $d=2$
$$
\sigma_m(\V(  Q_n)_1,\V_n^2)_\infty \ll n^{1/2}|Q_n|/m.
$$
\end{Lemma}
\begin{proof} 
 Take $f\in \V(  Q_n)$. Applying first Lemma \ref{L4.1v1} with $p=2$ and $[m/2]$ and, then, applying 
Theorem \ref{T2.6'} with $\Lambda =   Q_{n}$ and $[m/2]$ we obtain
$$
\sigma_m(f,\V_n^2)_\infty \ll n^{1/2}(| Q_n|/m)\|f\|_q,
$$
which proves the lemma.
\end{proof}
In the same way as above we derive the following result on the entropy numbers from Lemma \ref{L4.3v}.

\begin{Theorem}\label{T4.4v} We have for $d=2$
$$
\e_k(\Tr( Q_n)_1,L_\infty) \ll  \left\{\begin{array}{ll} n^{1/2}(| Q_n|/k) \log (4| Q_n|/k), &\quad k\le 2| Q_n|,\\
n^{1/2}2^{-k/(2| Q_n|)},&\quad k\ge 2| Q_n|.\end{array} \right.
$$
\end{Theorem}

\section{The upper bounds for function classes}

 In this section we provide a proof of the upper bounds in Theorems \ref{T2.1} -- \ref{T4.6}, except Theorem \ref{T1.5}, formulated in the Introduction. The proof uses Theorem \ref{T2.1G}, which is a general result on the   $\bar \bW^{a,b}_X\{X_n\}$. We apply this theorem in the case $X=L_q$, $1\le q\le \infty$, $X_n=\Tr(Q_n)$.   An important ingredient of Theorem \ref{T2.1G} is the assumption that subspaces $\{X_n\}$ satisfy (\ref{EA}). 
 The main work of this paper is devoted to establishing (\ref{EA}) in the case $X_n=\Tr(Q_n)$, $X=L_q$, $Y=L_p$, for different parameters $1\le q,p\le\infty$. We now indicate which results are used to obtain the appropriate versions of the (\ref{EA}) needed for the proof of the upper bounds in the corresponding theorem. 
 
 {\bf Proof of Theorem \ref{T2.1}.} The case $p=1$ follows from Remark \ref{R=}. The case $1<p<\infty$ follows from Theorem \ref{T4.2v}, which provides (\ref{EA}) with $\al=0$, $\bt=\ga=1-1/p$.
 
 {\bf Proof of Theorem \ref{T2.3}.}   It follows from Theorem \ref{T4.4v}, which provides (\ref{EA}) with $\al=1/2$, $\bt=\ga=1$.

{\bf Proof of Theorem \ref{T1.6}.}   It follows from Theorem \ref{T4.2vc}, which provides (\ref{EA}) with $\al=-1/2$, $\bt=\ga=1-1/p$ for $p\in [2,\infty)$.

 {\bf Proof of Theorem \ref{T1.7}.} It follows from Remark \ref{R=}. 
 
 {\bf Proof of Theorem \ref{T4.6}.}   It follows from Theorem \ref{T4.4}, which provides (\ref{EA}) with $\al=1/2$, $\bt=\ga=1/q$.

\section{Some properties of classes $\bW^{a,b}_q$ and $\bar\bW^{a,b}_q$}

  In this section we give some embedding type properties of classes $\bW^{a,b}_q$ and $\bar\bW^{a,b}_q$. As we already mentioned above we have the following two properties.

\begin{Proposition}\label{P5.1} For $1<q<\infty$ the classes $\bW^{a,b}_q$ and $\bar\bW^{a,b}_q$ are equivalent. 
\end{Proposition}

\begin{Proposition}\label{P5.2} For $1<q<\infty$ we have $\bW^r_{q,\alpha} \hookrightarrow \bW^{r,0}_q$. 
\end{Proposition}

The Nikol'skii type inequalities from Theorem \ref{NI} imply the following embeddings.

\begin{Proposition}\label{P5.3} For $1\le q<p<\infty$ and $a>\bt:=1/q-1/p$ we have $\bW^{a,b}_{q} \hookrightarrow \bW^{a-\bt,b}_p$ and $\bar\bW^{a,b}_{q} \hookrightarrow \bar\bW^{a-\bt,b}_p$. 
\end{Proposition}

It is well known (see \cite{Tmon} and \cite{TBook}) that for classes $\bW^r_{q,\alpha}$ and other classes with mixed smoothness approximation by the hyperbolic cross trigonometric polynomials $\Tr(Q_n)$ plays the same role as approximation by classical trigonometric polynomials plays in approximation of the univariate smoothness classes. There are two classical ways to describe compact sets, which are of interest in approximation theory and numerical analysis (see, for instance, \cite{D}). The first way to describe such classes uses the notion of smoothness. A typical example of such classes are classes $\bW^r_{q,\alpha}$. 
The second way to describe such classes is through approximation. In our context these would be the classes
$$
{\mathbf A}^{a,b}_q := \{f\in L_q: E_{Q_n}(f)_q \le 2^{-an}n^{(d-1)b},\quad n=1,2,\dots,\quad \|f\|_q\le 1\},
$$
where
$$
E_{Q_n}(f)_q := \inf_{t\in \Tr(Q_n)}\|f-t\|_q
$$ 
is the best approximation of $f$ by the hyperbolic cross polynomials from $\Tr(Q_n)$ in the $L_q$ norm. 

\begin{Proposition}\label{P5.4} For $1\le q\le\infty$ classes $\bar\bW^{a,b}_q$ and ${\mathbf A}^{a,b}_q$ are equivalent.
\end{Proposition}
\begin{proof} Let $f\in \bar\bW^{a,b}_q$. Then by the definition of this class $f$ has a representation
$$
f=\sum_{j=1}^\infty t_j, \quad t_j\in \Tr(Q_j), \quad \|t_j\|_q \le 2^{-aj} j^{b(d-1)}.
$$
Therefore,
$$
E_{Q_n}(f)_q \le \|\sum_{j>n} t_j\|_q \ll 2^{-an} n^{b(d-1)},
$$
which implies $\bar\bW^{a,b}_q \hookrightarrow {\mathbf A}^{a,b}_q$. 
Let now $f\in {\mathbf A}^{a,b}_q$ and let $u_n\in \Tr(Q_n)$ be such that
$$
\|f-u_n\|_q = E_{Q_n}(f)_q.
$$
Define
$$
t_1:=u_1,\quad t_n := u_n -u_{n-1},\quad n=2,3,\dots.
$$
Then 
$$
f=\sum_{n=1}^\infty t_n, \quad t_n\in \Tr(Q_n), 
$$
with
$$
\|t_1\|_q \le 2, \quad \|t_n\|_q \le 3 (2^{-an} n^{b(d-1)}),\quad n=2,3,\dots, 
$$
 which implies ${\mathbf A}^{a,b}_q \hookrightarrow \bar\bW^{a,b}_q$.
 \end{proof}
 
 Here is a nontrivial embedding type inequality, which is a direct corollary of Lemma 3.3 of Chapter 1 from \cite{Tmon}.
 
 \begin{Proposition}\label{P5.5} For $1\le q<p<\infty$ we have
 $$
 E_{Q_n}(f)_p \le C(q,p,d) \left(\sum_{j>n} 2^{pj\bt}E_{Q_n}(f)_q^p\right)^{1/p}.
 $$
 \end{Proposition}
 
 We formulate one more nontrivial inequality (see Lemma 4.3 in \cite{VT152}), which might be useful in the further study of 
 classes $\bW^{a,b}_1$ and $\bar\bW^{a,b}_1$.
 
 \begin{Lemma}\label{L5.1} Let $1<p<\infty$. For any $\epsilon >0$ there exists a constant $C(\epsilon,d,p)$ such that for each $t\in \Tr(Q_n)$ we have
$$
\sum_{\|\bs\|_1\le n}\|\delta_\bs(t)\|_p \le C(\epsilon,d,p)n^\epsilon 2^{\bt n} \|t\|_1,\quad \bt=1-1/p.
$$
\end{Lemma}


\begin{thebibliography}{9999}


\bibitem{Bass} R.F. Bass,  Probability estimates for multiparameter Brownian processes, {\em Ann. Probab.}, {\bf 16} (1988),     251--264.

\bibitem{Bel}  E.S. Belinsky,  Estimates of entropy numbers and Gaussian measures for classes of functions with bounded mixed derivative, {\em J. Approx. Theory}, {\bf 93} (1998),   114-127.

\bibitem{D} R.A. DeVore, Nonlinear approximation {\em Acta Numerica}
{\bf 7} (1998), 51--150.

\bibitem{DD}  Ding Dung,  Approximation of multivariate functions by means of harmonic analysis, {\em Hab. Dissertation} (1985), Moscow, MGU.

\bibitem{DTU} Ding Dung, V.N. Temlyakov, and T. Ullrich, Hyperbolic Cross Approximation, arXiv:1601.03978v1 [math.NA] 15 Jan 2016.

\bibitem{G}  E.D. Gluskin,  Extremal properties of orthogonal parallelepipeds and their applications to the geometry of Banach spaces,  
Mat. Sb., {\bf 136} (1988),  85--96.

\bibitem{H}  K. H{\" o}llig  Diameters of classes of smooth functions 
 in Quantitative Approximation, Academic Press, New York, 1980  163--176



\bibitem{K1}  B.S. Kashin,  On certain properties of the space of trigonometric
polynomials with the uniform norm,  Trudy Mat. Inst. Steklov,  {\bf145}  (1980),  111--116;
 English transl. in  Proc. Steklov Inst. Math.,  {\bf 145}  (1981).
 
\bibitem{KT1} B. S. Kashin and V. N. Temlyakov,  On best $m$-terms
approximations and the entropy of sets in the space $L^1$,  Mat. Zametki,
 {\bf 56}  (1994),  57--86; 
 English transl. in Math. Notes,  {\bf 56} (1994),  1137--1157.
 
 \bibitem{KTE2} B.S. Kashin and V.N. Temlyakov, Estimate of approximate characteristics for classes of functions with bounded mixed derivative, {\em Math. Notes}, {\bf 58} (1995), 1340--1342.

\bibitem{KaTe03} B.S. Kashin and V.N. Temlyakov, The volume estimates and their
applications, {\em East J. Approx.}, {\bf 9}  (2003), 469--485.




 \bibitem{KL}  J. Kuelbs and W.V. Li,   Metric entropy and the small ball problem for Gaussian measures, {\em   J. Functional Analysis}, {\bf  116} (1993),  133--157.
 
\bibitem{LiTs}  M.A. Lifshits and B.S. Tsirelson, Small deviations of Gaussian fields, {\em Teor. Probab. Appl.}, {\bf 31} (1986),      557--558.

\bibitem{M}  V.E. Maiorov  On various widths of the class $H^r_p$ in the space $L_q$  Izv. Akad. Nauk SSSR Ser. Mat.  {\bf 42} (1978),  773--788;
 English transl. in  Math. USSR-Izv.  {\bf 13} (1979).
 
 \bibitem{OO} D. Offin and K. Oskolkov, A note on orthonormal polynomial bases and wavelets, Constructive Approx. {\bf 9} (1993), 319--325.


  \bibitem{Schu}   C. Sch{\" u}tt, Entropy numbers of diagonal operators between symmetric Banach spaces,
{\em J. Approx. Theory }, {\bf 40} (1984), 121--128.

\bibitem{Smo} S.A. Smolyak,  The $\epsilon$-entropy of the classes $E^{\alpha k}_s(B)$ and $W^\alpha_s(B)$ in the metric $L_2$, {\em Dokl. Akad. Nauk SSSR}, {\bf 131} (1960),    30--33.

\bibitem{TaE}  M. Talagrand, The small ball problem for the Brownian sheet, {\em Ann. Probab.}, {\bf 22} (1994), 1331--1354.

\bibitem{Tem3} V.N Temlyakov, Approximation of periodic functions of several
variables with bounded mixed difference,  Mat. Sb., {\bf 133} (1980),
 65--85;
  English transl. in  Math. USSR Sbornik {\bf 41} (1982).


\bibitem{Tmon} V.N. Temlyakov, Approximation of functions with bounded mixed derivative, Trudy MIAN, {\bf 178} (1986), 1--112. English transl. in Proc. Steklov Inst. Math., {\bf 1} (1989).

\bibitem{TE1}  V.N. Temlyakov, On estimates of $\epsilon$-entropy and widths of classes of functions with bounded mixed derivative or difference, {\em Dokl. Akad. Nauk SSSR}, {\bf 301}   (1988),    288--291;
  English transl. in {\em Soviet Math. Dokl.}, {\bf 38},    
84--87.

\bibitem{TE2}  V.N. Temlyakov, Estimates of the asymptotic characteristics of classes of functions with bounded mixed derivative or difference, {\em Trudy Matem. Inst. Steklov}, {\bf 189}   (1989),    138--168;
  English transl. in {\em Proceedings of the Steklov Institute of Mathematics}, 1990,   Issue 4, 
161--197.  

\bibitem{T3}  V.N. Temlyakov,  Bilinear Approximation and Related Questions,  Trudy Mat. Inst. Steklov,  {\bf 194}  (1992),  229--248;
 English transl. in  Proc. Steklov Inst.
of Math.,  {\bf 4}  (1993),  245-265.

\bibitem{T4}  V.N. Temlyakov,  Estimates of Best Bilinear Approximations of Functions and
Approximation Numbers of Integral Operators,  Mat. Zametki, {\bf 51} (1992),  125--134;
 English transl. in  Math. Notes, {\bf 51}  (1992),
 510-517.
 
  \bibitem{TBook}  V.N. Temlyakov, {\em Approximation of periodic functions}, 
Nova Science Publishes, Inc., New York.,  1993.

\bibitem{TE3}  V.N. Temlyakov, An inequality for trigonometric polynomials and its application for estimating the entropy numbers, {\em J. Complexity}, 
{\bf 11}   (1995),     293--307.

\bibitem{TE5}  V.N. Temlyakov, Some inequalities for multivariate Haar polynomials, {\em East J. Approx.}, {\bf 1}  (1995),     61--72.

\bibitem{TE4}  V.N. Temlyakov, On two problems in the multivariate approximation, {\em East J. Approx.}, {\bf 4}  (1998),     505--514.

 


\bibitem{VT69} V.N. Temlyakov, Greedy Algorithms with Regards to Multivariate Systems with Special Structure, Constructive Approximation, {\bf 16} (2000), 399--425.

\bibitem{Tbook} V.N. Temlyakov, Greedy approximation, Cambridge University
Press, 2011.

\bibitem{VT138} V.N. Temlyakov, An inequality for the entropy numbers and its application,
J. Approximation Theory, {\bf 173} (2013), 110--121.

\bibitem{VT150} V.N. Temlyakov, Constructive sparse trigonometric approximation and other problems for functions with mixed smoothness, Matem. Sb. {\bf 206} (2015), 131--160;  arXiv: 1412.8647v1 [math.NA] 24 Dec 2014, 1--37. 

\bibitem{VT152} V.N. Temlyakov, Constructive sparse trigonometric approximation for functions with small mixed smoothness, arXiv:1503.0282v1 [math.NA] 1 Mar 2015.

\bibitem{Tr} H. Triebel, Bases in Function Spaces, Sampling, Discrepancy, Numerical Integration, European Mathematical Society, Germany, 2010.

\end{thebibliography}
\end{document}